\documentclass[reqno]{amsart}
\usepackage{stix2}
\usepackage{amssymb}

\newtheorem{theorem}{Theorem}
\newtheorem{lemma}[theorem]{Lemma}
\newtheorem{proposition}[theorem]{Proposition}
\newtheorem{corollary}[theorem]{Corollary}

\newtheorem{remark}[theorem]{Remark}

\newtheorem*{thma}{Theorem A}
\newtheorem*{thmb}{Theorem B}
\newtheorem*{thmc}{Theorem C}

\newcommand{\R}{{\mathbb R}}
\newcommand{\C}{{\mathbb C}}

\newcommand{\inr}{\int_{\R}}
\newcommand{\inc}{\int_{\C}}
\newcommand{\ess}{{\rm ess}}

\DeclareMathOperator{\re}{{\rm Re}\,}
\DeclareMathOperator{\im}{{\rm Im}\,}
\DeclareMathOperator{\tr}{{\rm tr}\,}

\begin{document}

\title[Toeplitz-type operators between Fock spaces]
{Two-parameter Toeplitz-type operators\\ between Fock spaces}

\author{Xingtang Dong}
\address{School of Mathematics, Tianjin University, Tianjin 300354, China}
\email{dongxingtang@163.com; dongxingtang@tju.edu.cn}

\author{Zipeng Wang}
\address{College of Mathematics and Statistics, Chongqing University,
Chongqing 401331, China}
\email{zipengwang2012@gmail.com; zipengwang@cqu.edu.cn}

\author{Kehe Zhu}
\address{Department of Mathematics and Statistics, SUNY, Albany, NY 12222, USA}
\email{kzhu@albany.edu}



\thanks{Research supported by NNSF of China (Grant Numbers
12271396, 12471116 and 12271328) and
2025CDJ-IAIS YB-004 (Chongqing University).}

\begin{abstract}
We introduce and study a two-parameter family of Toeplitz-type operators
$T^{(s,t)}_\varphi$ on the Fock space $F^p(\C)$ for $0<p\leq\infty$.
For all $0<p, q \leq\infty$ we obtain sharp conditions for $\varphi$ such that
$T^{(s,t)}_\varphi: F^p(\C)\to F^q(\C)$ is bounded or compact. In the
Hilbert space case $T^{(s,t)}_\varphi: F^2(\C)\to F^2(\C)$, we also study the
Berezin transform of $T^{(s,t)}_\varphi$ and membership of $T^{(s,t)}_\varphi$
in Schatten classes. Furthermore, we find a class of noncompact operators
whose Berezin transforms nevertheless vanish at infinity.
\end{abstract}

\maketitle

\section{Introduction}

Let $\C$ be the complex plane and $H(\C)$ be the linear space of all entire
functions on $\C$. For $0<p\leq\infty$, we use $F^p$ to denote the space of all
$f\in H(\C)$ such that the function $f(z)e^{-|z|^2/2}$ belongs to
$L^p(\C, dA)$, where $dA$ is ordinary area measure on the complex
plane. For $f\in F^p$ with $0<p< \infty$ we write
$$\|f\|_p^p=\frac{p}{2\pi}\inc \left|f(z)e^{-\frac{|z|^2}{2}}\right|^p\,dA(z).$$
For $f\in F^\infty$ we define
$$\|f\|_\infty=\sup\left\{|f(z)|e^{-\frac{|z|^2}{2}}: z\in\C\right\}.$$
The spaces $F^p$ are called Fock spaces.

When $p=2$, $F^2$ is a Hilbert space whose inner product is given by
$$\langle f,g\rangle=\inc f(z)\overline{g(z)}\,d\lambda(z),$$
where
$$d\lambda(z)=\frac1\pi e^{-{|z|^2}}\,dA(z)$$
is the standard Gaussian measure on $\C$. Furthermore, $F^2$ is a reproducing
kernel Hilbert space whose kernel function is $K(z,w)=e^{z\overline w}$.
See \cite{Zhu1} for an introduction to the theory of Fock spaces.

In a recent work \cite{DZ3}, the first and third authors introduced and studied
a two-parameter family of integral operators on $L^2(\C, d\lambda)$ and on
$F^2$. More specifically, for any $(s,t)\in\C^2$ with $s\not=0$, define
$$T^{(s,t)}f(z)=\inc K^{(s,t)}(z,w)f(w)\,d\lambda(w),$$
where
$$K^{(s,t)}(z,w)=\frac1{\sqrt s}\,\exp\left(\frac{tz^2-\overline{tw^2}
+2z\overline w}{2s}\right).$$
Here and throughout the paper, the complex square root is defined as follows:
$$\sqrt z=\sqrt{|z|}\,e^{i\theta/2},\qquad z=|z|e^{i\theta},
\qquad \theta\in(-\pi, \pi].$$
Under the inverse Bargmann transform, these operators include the classical
linear canonical transforms on the Lebesgue space $L^2(\R, dx)$ as special cases.
Thus we call these integral operators $T^{(s,t)}$ canonical integral operators
on the Fock space $F^2$.

One of the main results of \cite{DZ3} is that $T^{(s,t)}$ is bounded on $F^2$
if and only if $|s|^2\ge|t|^2+1$. Furthermore, $|s|^2=|t|^2+1$ if and only
if $T^{(s,t)}$ is unitary, and $|s|^2>|t|^2+1$ if and only if $T^{(s,t)}$ is
compact.

It is clear that the function $K^{(1,0)}(z,w)=e^{z\overline{w}}=K(z,w)$ is the
reproducing kernel of $F^2$ and the integral operator $T^{(1,0)}=P$ is the
orthogonal projection from $L^2(\C,d\lambda)$ to $F^2$. Recall that for
a ``reasonably good'' function $\varphi$ on $\C$ we can (densely) define an
operator $T_\varphi: H(\C)\to H(\C)$ by $T_\varphi f=P(\varphi f)$. This is
traditionally called the Toeplitz operator induced by $\varphi$ (or with symbol
$\varphi$). Since the operators $T^{(s,t)}$ include $P$ as a special case,
it is natural for us to introduce the two-parameter Toeplitz-type operators
$T_\varphi^{(s,t)}$ defined by $T_\varphi^{(s,t)}f=T^{(s,t)}(\varphi f)$.

It was shown in \cite{DZ4} that the operator $T^{(s,t)}$ is bounded on $F^p$
if and only if $|s|^2\geq |t|^2+1$, a condition that is actually independent of
$p$. Therefore, when considering $T_\varphi^{(s,t)}$ as operators on $F^p$,
we will make the natural assumption that $|s|^2\geq |t|^2+1$.

For any real number $\beta$ we use $L^\infty(\C, d\lambda_\beta)$ to
denote the space of Lebesgue measurable functions $f$ on $\C$ such that
$$\|f\|_{\infty,\beta}=\ess\sup\left\{|f(z)|\exp\left(-\beta|z|^2\right):
z\in\C\right\}<\infty.$$
When $\beta=0$, we simply write $dA$ for $d\lambda_\beta$.
In this paper we will study the boundedness, compactness, and membership
in Schatten-classes for $T_\varphi^{(s,t)}$ with symbols $\varphi\in
L^\infty(\C, d\lambda_{\beta})$.

Our first result provides upper bounds for the $F^p$ norm of
$T_\varphi^{(s,t)}K_z$, where $K_z$ denotes the reproducing kernel of
$F^2$ at $z$. These estimates are key to determining the boundedness,
compactness, and membership in Schatten classes of the operators
$T^{(s,t)}_\varphi$.

\begin{thma}
Suppose $0<p\le\infty$ and $(s, t)\in\C^2$ with $|s|^2\ge|t|^2+1$. If
$\varphi\in L^\infty(\C, d\lambda_{\beta_0})$ with
$$\beta_0=\frac{|s|^2-|t|^2-1}{2|s|(|s|+|t|)}\geq 0,$$
then there exists a positive constant $C$ such that
$$\left\|T_\varphi^{(s,t)}K_z\right\|_{p}
\leq C\exp\left(\frac{1}{2}|z|^2\right),\qquad z\in\C.$$
If $\varphi\in L^\infty(\C, d\lambda_\beta)$ with $\beta<\beta_0$,
then there exist constants $C>0$ and $\delta\in(0,1)$ such that
$$\left\|T_\varphi^{(s,t)}K_z\right\|_{p} \leq
C\exp\left(\frac{1-\delta}{2}|z|^2\right),\qquad z\in\C.$$
Moreover, the constant $\beta_0$ is sharp in the sense that
for any sufficiently small $\varepsilon>0$ there exist
$\varphi\in L^\infty(\C, d\lambda_{\beta_0+\varepsilon})$ and a
sequence $\{z_n\}$ such that
$$\lim_{|z_n|\rightarrow\infty}\ \exp\left(-\frac{1}{2}|z_n|^2\right)
\left\|T_\varphi^{(s,t)}K_{z_n}\right\|_{p}=\infty$$
and there exist $\varphi_0\in L^\infty(\C, d\lambda_{\beta_0})$ and
a sequence $\{z_n\}$ such that
$$\exp\left(-\frac{1}{2}|z_n|^2\right)\left\|T_{\varphi_0}^{(s,t)}K_{z_n}
\right\|_{p}\nrightarrow 0.$$
\end{thma}

Our next result concerns the boundedness, compactness and and membership
in Schatten classes for the operators $T^{(s,t)}_\varphi$. In the rest of
the paper, for convenience and simplicity of terminology, we will say that
a bounded linear operator $T: F^p\to F^q$, where $0<p,q\le\infty$, is
compact if $\|Tf_n\|_{q}\to0$ (as $n\to\infty$) for every bounded
sequence $\{f_n\}\subset F^p$ that converges to $0$ uniformly on
compact subsets of $\C$. It is well known that, for $1<p<\infty$, the space
$F^p$ is a reflexive Banach space and a sequence $\{f_n\}$ converges to $0$
weakly in $F^p$ if and only if the sequence is bounded in $F^p$ and
converges to $0$ uniformly on compact sets. So when $1<p<\infty$ and
$1\leq q\leq\infty$, this definition of compactness is equivalent to the usual
notion of compact operators between two Banach spaces.

\begin{thmb}
Let $(s,t)\in\C^2$ with $|s|^2\geq|t|^2+1$ and
$$\beta_0=\frac{|s|^2-|t|^2-1}{2|s|(|s|+|t|)}.$$
Then the operator $T_\varphi^{(s,t)}$ has the following properties.
\begin{enumerate}
\item[(a)] For any $\beta>\beta_0$ there exists a function
$\varphi\in L^\infty(\C, d\lambda_{\beta})$ such that $T_\varphi^{(s,t)}$
is unbounded from $F^p$ to $F^q$ for any $0<p,q\leq\infty$.
\item[(b)] If $\varphi\in L^\infty(\C,d\lambda_{\beta_0})$, then
$T_\varphi^{(s,t)}$ is bounded from $F^p$ to $F^q$ for all $0<p\leq q\leq\infty$.
\item[(c)] There exists a function $\varphi\in L^\infty(\C, d\lambda_{\beta_0})$
such that $T_\varphi^{(s,t)}$ is unbounded from $F^p$ to $F^q$ for all
$0<q<p\leq\infty$.
\item[(d)] There exists a function $\varphi\in L^\infty(\C, d\lambda_{\beta_0})$
such that $T_\varphi^{(s,t)}$ is not compact from $F^p$ to $F^q$ for any
$0< p\leq q\leq\infty$.
\item[(e)] If $\varphi\in L^\infty(\C, d\lambda_\beta)$ with $\beta<\beta_0$,
then $T_\varphi^{(s,t)}$ is compact from $F^p$ to $F^q$ for all
$0< p,q\leq\infty$.
\item[(f)] If $\varphi\in L^\infty(\C, d\lambda_\beta)$
with $\beta<\beta_0$, then on $F^{2}$, $T_\varphi^{(s,t)}$ belongs to the
Schatten $p$-class for any $p>0$. Additionally,
$${\rm tr } \left(T_\varphi^{(s,t)}\right)=\inc \varphi(z) K^{(s,t)}(z,z)
\,d\lambda(z).$$
\end{enumerate}
\end{thmb}

It should be noted that Theorem B implies that the constant $\beta_0$
is sharp for bounded (and compact) Toeplitz-type operators from $F^p$
to $F^q$ for all $0< p, q\leq\infty$. Also, Theorem B implies that the
constant $\beta_0$ constitutes a dividing point between bounded and
compact Toeplitz-type operators.

It is worth noting one special case, namely, Toeplitz-type operators
with bounded symbols on $F^p$ with $1\leq p\leq\infty$.
In this case, Theorem B implies that $T_\varphi^{(s,t)}$ is bounded for
any $(s,t)\in\C^2$ with $|s|^2\geq|t|^2+1$. Furthermore, we will obtain that
$$T_\varphi^{(s,t)}=T^{(s,t)}T_\varphi,$$
which reveals the relationship between our Toeplitz-type operators and
classical Toeplitz operators. As an application, we will determine exactly when
$T^{(s,t)}_\varphi$ is compact. See Proposition~\ref{prop:Toeplitz-old-new}
and Theorem~\ref{thmBSym}.

Now it is natural for us to consider the Berezin transform of
Toeplitz-type operators with bounded symbols on $F^2$.
Recall that for any bounded linear operator $T$ on $F^2$ the function
$$\widetilde T(z)=\langle Tk_z, k_z\rangle$$
is called the Berezin transform of $T$, where $k_z$ are the normalized
reproducing kernels of $F^2$. The following result provides interesting
information about the Berezin transform of our Toeplitz-type operators.

\begin{thmc}
Let $(s, t)\in\C^2$ with $|s|^2\geq|t|^2+1$. Then the Berezin transform
of $T^{(s,t)}_\varphi$ on $F^2$ have the following properties.
\begin{enumerate}
\item[(1)] If $s\neq 1\pm i|t|$, then for every $\varphi\in L^\infty(\C,dA)$,
there exist positive constants $C$ and $\delta$ such that
$$ \left|\widetilde{T_\varphi^{(s,t)}}(z)\right|\leq C \exp\left(-\delta|z|^2
\right), \qquad z\in\C.$$
\item[(2)] If $s=1\pm i|t|$, then there exists $\varphi\in L^\infty(\C,dA)$
(or specifically, $\varphi\equiv1$) and an increasing sequence $\{|z_n|\}$
tending to $\infty$ such that
$$\left|\widetilde{T_\varphi^{(s,t)}}(z_n)\right|=\frac{1}{\sqrt{|s|}}.$$
\end{enumerate}
\end{thmc}

Recall from Theorem B that if $\varphi\in L^\infty(\C, dA)$ and
$|s|^2>|t|^2+1$, then $T_\varphi^{(s,t)}$ is compact on $F^2$.
When a linear operator $T$ on $F^2$ is compact, its
Berezin transform has the property that $\langle Tk_z, k_z\rangle\to0$
as $z\to\infty$, because $k_z\to0$ weakly in $F^2$ as $z\to\infty$.
The first estimate in Theorem C clearly confirms this and gives more
information: when $T_\varphi^{(s,t)}$ is compact on $F^2$, its Berezin
tranform tends to $0$ very rapidly (i.e. second-order exponentially) as
$z\to\infty$.

More interesting (and somewhat surprising) is the case when
$\varphi\in L^\infty(\C, dA)$ and $|s|^2=|t|^2+1$. In this case,
Theorem~\ref{thmBSym} shows that $T_\varphi^{(s,t)}$ is compact on
$F^2$ (even on $F^p$ for any $1\leq p\leq \infty$) if and only if the
Berezin transform of the symbol function $\varphi$ satisfies the following:
$$\lim_{|z|\to\infty} \widetilde{\varphi}(z)= 0.$$
Note that here $\widetilde\varphi(z)=\langle T_\varphi k_z, k_z\rangle$ is
the Berezin transform of the classical Toeplitz operator $T_\varphi$.
Therefore, by part (1) of Theorem C, if $|s|^2=|t|^2+1$ with
$s\neq 1\pm i|t|$ and if $\widetilde\varphi(z)\not\to0$ as $z\to\infty$
(for example, $\varphi=1$), the operator $T_\varphi^{(s,t)}$ is not
compact but its Berezin transform still tends to $0$ second-order
exponentially as $z\to\infty$.

\section{The $F^p$ norm of $T_\varphi^{(s,t)}k_z$}

In this section we seek conditions on the symbol $\varphi$ that ensure
the $F^p$ norm of $T_\varphi^{(s,t)}k_z$ is uniformly bounded for all
$z\in\C$, where $k_z=K_z/\|K_z\|_{F^2}$ is the normalized reproducing
kernel of $F^2$ at $z$. Each $k_w$ is a unit vector in $F^2$; it is also a
unit vector in $F^p$ for any $0<p\le\infty$. See \cite[Theorem 2.7]{Zhu1}.

We begin with some basic computations. Recall from Equation (1.18)
of \cite{B1} that
\begin{equation}\label{eqIC}
\inc e^{\frac{1}{2}\gamma w^{2}+aw}e^{\frac{1}{2}
\overline{\delta w^2}+\overline{bw}}\,d\lambda(w)
=\frac{1}{\sqrt{1-\gamma\overline{\delta}}}
\exp\left(\frac{\overline{\delta}a^{2}+\gamma \overline{b}^{2}
+2a\overline{b}}{2(1-\gamma\overline{\delta})}\right)
\end{equation}
under the assumption $|\gamma+\delta|^{2}<4$, which implies that
$$\re(1-\gamma\overline\delta)=1-\frac14|\gamma+\delta|^2+\frac14|
\gamma-\delta|^2>0.$$
For $|t|<2|s|$, a direct calculation using \eqref{eqIC} shows that
\begin{equation}\label{eqTKw}
T^{(s,t)}K_w(z)=K_w^{(s,t)}(z), \qquad z, w\in \C.
\end{equation}
See Lemma 4 of \cite{DZ3}. We then obtain the following basic result.

\begin{lemma}\label{lemiTK}
Let $(s, t)\in\C^2$ and let $\alpha>|t|/(2|s|)$. Then for all $z,w\in\C$ we have
$$\inc  K^{(s,t)}(w, v)K_z(v) \exp\left(-\alpha|v|^2\right)\,dA(v)
=\frac{\pi}{{\alpha }}\,K^{(s,t)}(w, z/{\alpha})$$
and
\begin{multline*}
\inc\left|K^{(s,t)}(w, v)K_z(v)\right| \exp\left(-\alpha|v|^2\right)\,dA(v)\\
=\frac{2\pi\sqrt{|s|}}{\sqrt{4\alpha^2|s|^2-|t|^2}}
\left|\exp\left(\frac{\overline{s}t}{2|s|^2}w^2-\frac{\overline{s}t
\left({w}+{s}{z}\right)^{2}}{2|s|^2(4\alpha^2|s|^2-|t|^2)}\right)\right|
\exp\left(\frac{\alpha\left|{w}+{s}{z}\right|^2}{4\alpha^2|s|^2-|t|^2}\right).
\end{multline*}
\end{lemma}

\begin{proof}
With the change of variables $u=\sqrt{\alpha}\,v$, we obtain
\begin{align*}
&\inc  K^{(s,t)}(w, v)K_z(v) \exp\left(-\alpha|v|^2\right)\,dA(v)\\
&\qquad=\frac{\pi}{\alpha\sqrt{s}}\inc  \exp\left(\frac{t w^2}{2 {s}}
-\frac{\overline t \overline{u}^2}{2{\alpha}{s}}+\frac{w\overline{u}}
{\sqrt{\alpha}{s}}+\frac{u\overline{z}}{\sqrt{\alpha}}\right)
\,d\lambda(u)\\[8pt]
&\qquad=\frac{\pi}{\sqrt{\alpha }}T^{(\alpha s,t)}
K_{z/\sqrt{\alpha}}(\sqrt{\alpha}w).
\end{align*}
Since $|t|<2\alpha|s|$, it follows from \eqref{eqTKw} that
$$T^{(\alpha s,t)}K_{z/\sqrt{\alpha}}(\sqrt{\alpha}w)=
K^{(\alpha s,t)}(\sqrt{\alpha}w, z/\sqrt{\alpha})=
\frac{1}{\sqrt{\alpha }} K^{(s,t)}(w, z/{\alpha}).$$
The desired result for the first equality then follows.

Similarly, with the change of variables $u=\sqrt{\alpha}\,v$, we obtain
\begin{multline*}
\inc\left|K^{(s,t)}(w, v)K_z(v)\right|\exp\left(-\alpha|v|^2\right)\,dA(v)\\
=\frac{\pi}{\alpha\sqrt{|s|}}\left|\exp\left(\frac{\overline{s}t}
{2|s|^2}w^2\right)\right|\inc \left|\exp\left(\frac{-t}{2\alpha
\overline{s}}u^2+\frac{ u\overline{w}}{\sqrt{\alpha}\overline{s}}
+\frac{u\overline{z}}{\sqrt{\alpha}}\right)\right|\,d\lambda(u).
\end{multline*}
Since $|t|<2\alpha|s|$, it follows from \eqref{eqIC} that
\begin{align*}
&\inc\left|\exp\left(\frac{-t}{4\alpha \overline{s}}u^2+
\frac{ u\overline{w}}{2\sqrt{\alpha}\overline{s}}+
\frac{u\overline{z}}{2\sqrt{\alpha}}\right)\right|^2\,d\lambda(u)\\
&\qquad=\frac{2\alpha|s|}{\sqrt{4\alpha^2|s|^2-|t|^2}}
\left|\exp\left(\frac{-\overline{s}t\left({w}+{s}{z}\right)^{2}}{2|s|^2
(4\alpha^2|s|^2-|t|^2)}\right)\right|\exp\left(\frac{\alpha\left|{w}
+{s}{z}\right|^2}{4\alpha^2|s|^2-|t|^2}\right),
\end{align*}
from which we obtain the second desired equality.
\end{proof}

For any $\alpha>0$ and $(s,t)\in\C^2$ with $s\not=0$ we define
$$Q^{(s,t)}_\alpha f(z)=\inc \left|K^{(s,t)}(z,w)\right|f(w)
\exp\left(-\alpha|v|^2\right)\,dA(w).$$
The following result is critical for the main result of this section.

\begin{proposition}\label{propNTKap}
Let $(s, t)\in\C^2$ with $|s|^2\geq |t|^2+1$, $0<p<\infty$, and
$\alpha>1/2$. Then for all $z\in\C$ we have
\begin{multline*}
\inc\left[Q^{(s,t)}_\alpha \left(\left|K_z\right|\right)(w)\right]^p
\exp\left(-\frac{p|{w}|^2}{2}\right)\,dA(w)\\
= C\,\left|\exp\left(\frac{-p{s}t(|s|^2-|t|^2-1){z}^2}
{2c_{s,t}(\alpha)}\right)\right|\exp\left(\frac{p|s|^2(2\alpha|s|^2
-2\alpha|t|^2-1)|{z}|^2}{2c_{s,t}(\alpha)}\right),
\end{multline*}
where
$$c_{s,t}(\alpha)=(|s|^2-|t|^2-1)(4\alpha^2|s|^2-|t|^2)+
4\alpha^2 |s|^2-4\alpha |s|^2+|t|^2+1>0$$
and
$$C=\frac{(2\pi)^{1+p}|s|^{1+p/2}(4\alpha^2|s|^2
-|t|^2)^{(1-p)/2}}{p\sqrt{c_{s,t}(\alpha)}}.$$
\end{proposition}

\begin{proof}
Since $\alpha>1/2>|t|/(2|s|)$, it follows from
Lemma \ref{lemiTK} that
$$Q^{(s,t)}_\alpha\left(\left|K_z\right|\right)(w)=
\frac{2\pi\sqrt{|s|}}{\sqrt{4\alpha^2|s|^2-|t|^2}}f(z)g_z(w),$$
where
$$f(z):=\left|\exp\left(\frac{-{s}{t}z^{2}}{2(4\alpha^2|s|^2-
|t|^2)}\right)\right|\exp\left(\frac{\alpha|s|^2|z|^2}{4\alpha^2
|s|^2-|t|^2}\right)$$
and
$$g_z(w):=\left|\exp\left(\frac{\overline{s}{t}(4\alpha^2|s|^2
-|t|^2-1)w^{2}}{2|s|^2(4\alpha^2|s|^2-|t|^2)}+
\frac{(2\alpha\overline{s}\overline{z}-{t}z) w}{4\alpha^2|s|^2
-|t|^2}\right)\right|\exp\left(\frac{\alpha|{w}|^2}{4\alpha^2|s|^2
-|t|^2}\right).$$
It is easy to check that
$$4\alpha^2|s|^2-|t|^2-2\alpha>0,\qquad \alpha>1/2.$$
Then by a simple change of variables, we see that{\small
\begin{align*}
&\inc\left[g_z(w)\right]^p\exp\left(\frac{-p|{w}|^2}{2}\right)\,dA(w)\\
&=C_1\inc\left|\exp\left(\frac{\overline{s}{t}(4\alpha^2|s|^2-|t|^2-1)
w^{2}}{|s|^2(4\alpha^2|s|^2-|t|^2-2\alpha)}+
\frac{\sqrt{2p}(2\alpha \overline{s}\overline{z}-{t}z)w}
{\sqrt{(4\alpha^2|s|^2-|t|^2)(4\alpha^2|s|^2-|t|^2-2\alpha)}}
\right)\right|\,d\lambda(w),
\end{align*}
}where
$$C_1=2\pi(4\alpha^2|s|^2-|t|^2)/[p(4\alpha^2|s|^2-|t|^2-2\alpha)].$$

Let us rewrite $c_{s,t}(\alpha)$ as a quadratic function of $\alpha$,
\begin{equation}\label{eqcst}
c_{s,t}(\alpha)=4|s|^2(|s|^2-|t|^2)\alpha^2-
4|s|^2 \alpha-(|s|^2-|t|^2)|t|^2+2|t|^2+1.
\end{equation}
Since
$$\frac{1}{2(|s|^2-|t|^2)}\leq \frac{1}{2},\qquad c_{s,t}\left(\frac{1}{2}
\right)=(|s|^2-|t|^2-1)^2\geq 0,$$
we must have
\begin{equation}\label{eqfaiphag0}
c_{s,t}(\alpha)>0, \qquad \alpha>1/2.
\end{equation}
A few lines of elementary calculations show that
\begin{equation*}
{|s|^2(4\alpha^2|s|^2-|t|^2-2\alpha)^2-|t|^2(4\alpha^2|s|^2
-|t|^2-1)^2}=(4\alpha^2|s|^2-|t|^2)c_{s,t}(\alpha).
\end{equation*}
Combining this with \eqref{eqfaiphag0}, we see that
$$1-\frac{|t|^2(4\alpha^2|s|^2-|t|^2-1)^2}{|s|^2(4\alpha^2|s|^2
-|t|^2-2\alpha)^2}=\frac{(4\alpha^2|s|^2-|t|^2)c_{s,t}(\alpha)}
{|s|^2(4\alpha^2|s|^2-|t|^2-2\alpha)^2}>0.$$
Then by \eqref{eqIC}, we have
\begin{align*}
&\inc\left|\exp\left(\frac{\overline{s}{t}(4\alpha^2|s|^2-|t|^2-1)
w^{2}}{2|s|^2(4\alpha^2|s|^2-|t|^2-2\alpha)}+\frac{\sqrt{p}
(2\alpha \overline{s}\overline{z}-{t}z)w}{\sqrt{2(4\alpha^2|s|^2
-|t|^2)(4\alpha^2|s|^2-|t|^2-2\alpha)}}\right)\right|^2
d\lambda(w)\\[8pt]
&=C_2\,\left|\exp\left(\frac{p s\overline{t}(4\alpha^2|s|^2
-|t|^2-1)(2\alpha \overline{s}\overline{z}-{t}z)^2}
{2(4\alpha^2|s|^2-|t|^2)^2c_{s,t}(\alpha)}\right)\right|\\[8pt]
&\qquad\qquad\cdot\exp\left(\frac{p|s|^2(4\alpha^2|s|^2-|t|^2
-2\alpha)|2\alpha \overline{s}\overline{z}-{t}z|^2}
{2(4\alpha^2|s|^2-|t|^2)^2c_{s,t}(\alpha)}\right),
\end{align*}
where
$$C_2=\frac{|s|(4\alpha^2|s|^2-|t|^2-2\alpha)}
{\sqrt{(4\alpha^2|s|^2-|t|^2)c_{s,t}(\alpha)}}.$$
A direct calculation shows that
\begin{align*}
&({s}{t}|t|^2 +4\alpha^2{s}{t}|s|^2)(4\alpha^2|s|^2-|t|^2-1)
-4\alpha{s}{t}|s|^2(4\alpha^2|s|^2-|t|^2-2\alpha)\\
&\qquad\qquad={s}{t} (4\alpha^2|s|^2-|t|^2)(4\alpha^2 |s|^2
-4\alpha |s|^2+|t|^2+1)
\end{align*}
and
\begin{align*}
&|s|^2(|{t}|^2+4\alpha^2|s|^2)(4\alpha^2|s|^2-|t|^2
-2\alpha)-4\alpha|s|^2|t|^2(4\alpha^2|s|^2-|t|^2-1)\\
&\qquad\qquad=|s|^2 (4\alpha^2|s|^2-|t|^2)(4\alpha^2 |s|^2
-4\alpha |t|^2+|t|^2-2\alpha).
\end{align*}
We arrive at
\begin{align*}
\inc&\left[g_z(w)\right]^p\exp\left(\frac{-p|{w}|^2}{2}
\right)\,dA(w)\\[8pt]
&= C_1C_2\, \left|\exp\left(\frac{p{s}t(4\alpha^2 |s|^2
-4\alpha |s|^2+|t|^2+1)}{2(4\alpha^2|s|^2-|t|^2)
c_{s,t}(\alpha)}{z}^2\right)\right|\\[8pt]
&\quad\qquad\cdot\exp\left(\frac{p|s|^2(4\alpha^2|s|^2
-4\alpha |t|^2+|t|^2-2\alpha)}{2(4\alpha^2|s|^2-|t|^2)
c_{s,t}(\alpha)}|{z}|^2\right).
\end{align*}
Consequently,
\begin{align*}
\inc&\left[Q^{(s,t)}_\alpha \left(\left|K_z\right|\right)(w)\right]^p
\exp\left(-\frac{p|{w}|^2}{2}\right)\,dA(w)\\[8pt]
&=\left[\frac{2\pi\sqrt{|s|}}{\sqrt{4\alpha^2|s|^2-|t|^2}}
\right]^p[f(z)]^p\,\inc \left[g_z(w)\right]^p
\exp\left(\frac{-p|{w}|^2}{2}\right)\,dA(w)\\[8pt]
&=C\, \left|\exp\left(\frac{-p{s}t(|s|^2-|t|^2-1){z}^2}
{2c_{s,t}(\alpha)}\right)\right|\exp\left(\frac{p|s|^2
(2\alpha|s|^2-2\alpha|t|^2-1)|{z}|^2}{2c_{s,t}(\alpha)}\right).
\end{align*}
This proves the desired result.
\end{proof}

We are now ready to prove the main result of this section.

\begin{theorem}\label{thmTanorml}
Let $(s, t)\in\C^2$ with $|s|^2\ge|t|^2+1$, $0<p\leq\infty$, and
$\varphi\in L^\infty(\C, d\lambda_{\beta})$ with $\beta<1/2$.
Then for any $z\in\C$,
$$\exp\left(-\frac{1}{2}|z|^2\right)\left\|T_\varphi^{(s,t)}
K_z\right\|_{p}\leq C_{\beta} \exp\left(-\frac{H(\beta)}
{2c_{s,t}(1-\beta)}|{z}|^2\right),$$
where $c_{s,t}(\cdot)$ is the quadratic function defined in
Proposition \ref{propNTKap},
$$H(\beta)=4|s|^2(|s|^2-|t|^2)\beta^2-2|s|^2(3|s|^2
-3|t|^2-2)\beta+(|s|^2-|t|^2-1)(2|s|^2-|s||t|-|t|^2-1),$$
and
$$C_{\beta}=\frac{2\|\varphi\|_{\infty,\beta}\sqrt{|s|}}
{\sqrt{4(1-\beta)^2|s|^2-|t|^2}}\left(\frac{|s|\sqrt{4(1-\beta)^2
|s|^2-|t|^2}}{\sqrt{c_{s,t}(1-\beta)}}\right)^{1/p}.$$
\end{theorem}

\begin{proof}
We first assume $0<p<\infty$. For $\varphi\in
L^\infty(\C, d\lambda_{\beta})$ with $\beta<1/2$, it is clear that
\begin{align*}
&\left\|T_\varphi^{(s,t)}K_z\right\|_{p}^p\\
&\quad\leq \frac{p}{2\pi}\inc\left[\inc\left|K^{(s,t)}(w,v)
\right||\varphi(v)| |K_z(v)|d\lambda(v)\right]^p
\exp\left(-\frac{p|{w}|^2}{2}\right)\,dA(w)\\
&\quad\leq\frac{p\|\varphi\|_{\infty,\beta}^p}{2\pi^{1+p}}
\inc\left[Q^{(s,t)}_\alpha \left(\left|K_z\right|\right)(w)\right]^p
\exp\left(-\frac{p|{w}|^2}{2}\right)\,dA(w),
\end{align*}
where $\alpha=1-\beta>1/2$. Note that
\begin{multline*}
\frac{|s|^2(2\alpha|s|^2-2\alpha|t|^2-1)+|s||t|(|s|^2-|t|^2-1)}
{2c_{s,t}(\alpha)}-\frac{1}{2}\\[8pt]
=-\frac{(|s|^2-|t|^2-1)(4\alpha^2|s|^2-2\alpha|s|^2
-|s||t|-|t|^2-1)+2(2\alpha^2-3\alpha+1)|s|^2}{2c_{s,t}(\alpha)}.
\end{multline*}
By Proposition \ref{propNTKap}, we have{\small
\begin{multline*}
\exp\left(-\frac{1}{2}|z|^2\right)\left\|T_\varphi^{(s,t)}
K_z\right\|_{p}\\[8pt]
\leq  C_{\beta}\exp\left(-\frac{(|s|^2-|t|^2-1)(4\alpha^2|s|^2
-2\alpha|s|^2-|s||t|-|t|^2-1)+2(2\alpha^2-3\alpha+1)|s|^2}
{2c_{s,t}(\alpha)}|{z}|^2\right).
\end{multline*}
}Since
\begin{multline*}
(|s|^2-|t|^2-1)(4\alpha^2-2\alpha)+2(2\alpha^2-3\alpha+1)\\
=(2\alpha-2)\left[(|s|^2-|t|^2)(2\alpha-2)+(3|s|^2-3|t|^2-2)
\right]+2(|s|^2-|t|^2-1)
\end{multline*}
and $\beta=1-\alpha$, we have
$$\exp\left(-\frac{1}{2}|z|^2\right)\left\|T_\varphi^{(s,t)}
K_z\right\|_{p}\leq C_{\beta} \,\exp\left(-\frac{H(\beta)}
{2c_{s,t}(1-\beta)}|{z}|^2\right).$$
This proves the desired result for the case $0<p<\infty$.

Next we assume $p=\infty$. Since $T_\varphi^{(s,t)}K_z\in F^q$
for any $0<q<\infty$, it is clear from Corollary 2.8 of \cite{Zhu1} that
$$\left\|T_\varphi^{(s,t)}K_z\right\|_{\infty}\leq
\left\|T_\varphi^{(s,t)}K_z\right\|_{q}.$$
Combining this with what has been proved, we obtain the
desired results for $p=\infty$.
\end{proof}

The following lemma tells us how the constant $\beta_{0}$ in
Theorem A is determined.

\begin{lemma}\label{lemquadr}
Let $(s, t)\in\C^2$ with $|s|^2\ge|t|^2+1$ and let $H(\beta)$ be the
quadratic function defined in Theorem~\ref{thmTanorml}. Then
$H(\beta)\geq0$ on $(-\infty,1/2)$ if and only if
$$\beta\leq\beta_{0}:=\frac{|s|^2-|t|^2-1}{2|s|(|s|+|t|)}.$$
\end{lemma}

\begin{proof}
Clearly, the solutions of the quadratic equation $H(\beta)=0$ are given by
$$\beta_{0}=\frac{|s|^2-|t|^2-1}{2|s|(|s|+|t|)},\qquad
\beta_{1}=\frac{2|s|^2-|s||t|-|t|^2-1}{2|s|(|s|-|t|)}.$$
Since $0\leq\beta_0<1/2\leq\beta_{1}$ and
$$H(0)=(|s|^2-|t|^2-1)[|s|(|s|-|t|)+|s|^2-|t|^2-1]\geq 0,$$
we conclude that $H(\beta)\geq 0$ on $(-\infty,1/2)$ if and only
if $\beta\leq\beta_0$.
\end{proof}

We are now ready to prove Theorem A.

\begin{proof}[Proof of Theorem A]
We first assume $\varphi\in L^\infty(\C, d\lambda_{\beta_0})$. The proof of
Lemma~\ref{lemquadr} shows that $H(\beta_0)=0$. It then follows from Theorem~\ref{thmTanorml} that
$$\exp\left(-\frac{1}{2}|z|^2\right)\left\|T_\varphi^{(s,t)}K_z\right\|_{p}
\leq C_{\beta_0}, \qquad 0<p\leq \infty,\quad z\in\C.$$

Now assume $\varphi\in L^\infty(\C, d\lambda_{\beta})$ with
$\beta<\beta_0$. It is clear from the proof of Lemma~\ref{lemquadr} that
$H(\beta)>0$ for $\beta<\beta_0$. Combining this with
Theorem~\ref{thmTanorml}, we obtain
$$\exp\left(-\frac{1}{2}|z|^2\right)\left\|T_\varphi^{(s,t)}K_z\right\|_{p}
\leq C_{\beta} \exp\left(-\frac{\delta}{2}|{z}|^2\right),$$
where $\delta={H(\beta)}/{c_{s,t}(1-\beta)}$.
Recall from \eqref{eqcst} that
$$c_{s,t}(1-\beta)=4|s|^2(|s|^2-|t|^2)(1-\beta)^2-4|s|^2
(1-\beta)-(|s|^2-|t|^2)|t|^2+2|t|^2+1.$$
Since $\beta<\beta_0<1/2$ and $|s|^2\geq|t|^2+1$, an elementary
calculation shows
\begin{align*}
H(\beta)&-c_{s,t}(1-\beta)\\
&=2|s|^2(|s|^2-|t|^2)\beta-|s||t|(|s|^2-|t|^2-1)
+|s|^2-2|s|^2(|s|^2-|t|^2)\\
&< -(|s|^2+|s||t|)(|s|^2-|t|^2-1)\leq0,
\end{align*}
which implies $\delta\in(0,1)$.

The sharpness of the constant $\beta_0$ in Theorem A will be shown in Section 4. See Remark~\ref{remsharp} for more details.
This
completes the proof of the theorem.
\end{proof}

\section{Boundedness and compactness of Toeplitz-type operators}

In this section, we consider the boundedness and compactness of the
Toeplitz-type operators $T_{\varphi}^{(s,t)}$ from $F^p$ to $F^q$
for all $0< p, q \leq\infty$. Our analysis is based on estimates
and computations with the integral kernel of $T_{\varphi}^{(s,t)}$.
We first need the following critical integral estimate.

\begin{lemma}\label{lemexpc2}
Let $(s, t)\in\C^2$ with $|s|^2\geq |t|^2+1$,
$$\gamma_0:=\frac{|s|^2-|t|^2-1}{2|s|\left(|s|
+\sqrt{|t|^2+1}\right)}\geq0,$$
and $\varphi\in L^\infty(\C, d\lambda_\gamma)$ with
$\gamma<(|s|-|t|)/(2|s|)$. Then
\begin{multline*}
\inc\left|K^{(s,t)}(z,w)\right|\,|\varphi(w)|
\exp\left(-\frac{|z|^2+|w|^2}{2}\right)\,dA(w)\\
\leq C_{\gamma}\, \|\varphi\|_{\infty,\gamma}
\exp\left(-\frac{G(\gamma)}{2|s|\left[(1-2\gamma)^2|s|^2
-|t|^2\right]}|{z}|^2\right),
\end{multline*}
where $C_\gamma={2\pi\sqrt{|s|}}/{\sqrt{(1-2\gamma)^2|s|^2-|t|^2}}$,
$$G(\gamma)=4|s|^2(|s|-|t|)\gamma^2-2|s|(2|s|^2-2|s||t|-1)\gamma
+(|s|^2-|t|^2-1)(|s|-|t|)$$
for $\gamma\leq \gamma_0$, and
$$G(\gamma)=4|s|^2(|s|+|t|)\gamma^2-2|s|(2|s|^2+2|s||t|-1)
\gamma+(|s|^2-|t|^2-1)(|s|+|t|)$$
for $\gamma_0<\gamma< (|s|-|t|)/(2|s|)$.
\end{lemma}

\begin{proof}
Let $\alpha=1/2-\gamma$. Then $\gamma<(|s|-|t|)/(2|s|)$ implies
$\alpha>{|t|}/{(2|s|)}$. By Lemma~\ref{lemiTK},
\begin{multline*}
\inc\left|K^{(s,t)}(z, w)\right|\exp\left(-\alpha|w|^2\right)\,dA(w)\\
=C_\gamma \,\left|\exp\left(\frac{\overline{s}{t}(4\alpha^2|s|^2
-|t|^2-1)z^{2}}{2|s|^2(4\alpha^2|s|^2-|t|^2)}\right)\right|
\exp\left(\frac{\alpha|{z}|^2}{4\alpha^2|s|^2-|t|^2}\right).
\end{multline*}

First consider the case when $\gamma\leq\gamma_0$. Then
$\alpha\geq{\sqrt{|t|^2+1}}/{(2|s|)}$, which implies that
$4\alpha^2|s|^2-|t|^2-1\geq0$. Therefore,
\begin{multline}\label{eqc1}
\exp\left(\frac{-|{z}|^2}{2}\right)\inc\left|K^{(s,t)}(z, w)\right|
\exp\left(-\alpha|w|^2\right)\,dA(w)\\
\leq C_\gamma\,\exp\left(-\frac{|s|(4\alpha^2|s|^2-|t|^2-2\alpha)
-|t|(4\alpha^2|s|^2-|t|^2-1)}{2|s|(4\alpha^2|s|^2-|t|^2)}|{z}|^2\right).
\end{multline}
Note that
\begin{multline*}
|s|(4\alpha^2|s|^2-|t|^2-2\alpha)-|t|(4\alpha^2|s|^2-|t|^2-1)\\
=4|s|^2(|s|-|t|)\gamma^2-2|s|(2|s|^2-2|s||t|-1)\gamma+(|s|^2
-|t|^2-1)(|s|-|t|)
\end{multline*}
Consequently, for any $\varphi\in L^\infty(\C, d\lambda_{\gamma})$,
\eqref{eqc1} yields
\begin{align*}
&\inc\left|K^{(s,t)}(z, w)\right||\varphi(w)|
e^{-\frac{|z|^2+|w|^2}{2}}dA(w)\\
&\qquad\leq\|\varphi\|_{\infty,\gamma}\exp\left(\frac{-|{z}|^2}{2}
\right)\inc\left|K^{(s,t)}(z, w)\right|\exp\left(-\frac{1-2\gamma}{2}
|w|^2\right)\,dA(w)\\
&\qquad\leq  C_{\gamma}\,\|\varphi\|_{\infty,\gamma}
\exp\left(-\frac{G(\gamma)}{2|s|\left[(1-2\gamma)^2|s|^2
-|t|^2\right]}|{z}|^2\right).
\end{align*}

Next consider the case when $\gamma_0<\gamma< (|s|-|t|)/(2|s|)$.
Then $$0\leq \frac{|t|}{2|s|}<\alpha<\frac{\sqrt{|t|^2+1}}{2|s|},$$
which implies that $4\alpha^2|s|^2-|t|^2-1<0$. Note that
\begin{multline*}
|s|(4\alpha^2|s|^2-|t|^2-2\alpha)+|t|(4\alpha^2|s|^2-|t|^2-1)\\
=4|s|^2(|s|+|t|)\gamma^2-2|s|(2|s|^2+2|s||t|-1)\gamma
+(|s|^2-|t|^2-1)(|s|+|t|)
\end{multline*}
Then a similar argument shows that
\begin{align*}
&\inc\left|K^{(s,t)}(z, w)\right||\varphi(w)|
e^{-\frac{|z|^2+|w|^2}{2}}dA(w)\\
&\qquad\leq C_\gamma\, \|\varphi\|_{\infty,\gamma}
\exp\left(-\frac{|s|(4\alpha^2|s|^2-|t|^2-2\alpha)+|t|(4\alpha^2|s|^2
-|t|^2-1)}{2|s|(4\alpha^2|s|^2-|t|^2)}|{z}|^2\right)\\
&\qquad\leq  C_{\gamma}\, \|\varphi\|_{\infty,\gamma}
\exp\left(-\frac{G(\gamma)}{2|s|\left[(1-2\gamma)^2|s|^2-|t|^2
\right]}|{z}|^2\right).
\end{align*}
This proves the desired result.
\end{proof}

The following lemma determines the constant $\beta_{0}$ in Theorem B.

\begin{lemma}\label{lemquadrG}
Let $(s, t)\in\C^2$ with $|s|^2\ge|t|^2+1$ and let
$G(\gamma)$ be the piecewise quadratic function defined in
Lemma~\ref{lemexpc2}. Then $G(\gamma)\geq0$ on
$\left(-\infty,(|s|-|t|)/(2|s|)\right)$ if and only if
$$\gamma\leq\beta_{0}:=\frac{|s|^2-|t|^2-1}{2|s|(|s|+|t|)}.$$
\end{lemma}

\begin{proof}
Let $\gamma_0$ be the number from Lemma~\ref{lemexpc2} and
consider the first form of the piecewise quadratic function
$$G(\gamma)=4|s|^2(|s|-|t|)\gamma^2-2|s|(2|s|^2-2|s||t|-1)\gamma
+(|s|^2-|t|^2-1)(|s|-|t|).$$
The solutions of the quadratic equation $G(\gamma)=0$ are given by
$$\gamma_{1}=\frac{|s|^2-|t|^2-1}{2|s|(|s|-|t|)}\geq0,
\qquad\gamma_{2}=\frac{|s|-|t|}{2|s|}>0.$$
Note that $G(0)=(|s|^2-|t|^2-1)(|s|-|t|)\geq 0$ and
$\gamma_0\leq\min\left\{\gamma_{1}, \gamma_{2}\right\}$.
So in this case, we have $G(\gamma)\geq 0$ for all
$\gamma\leq \gamma_0$.

Next consider the second form of the piecewise quadratic function
$$G(\gamma)=4|s|^2(|s|+|t|)\gamma^2-2|s|(2|s|^2
+2|s||t|-1)\gamma+(|s|^2-|t|^2-1)(|s|+|t|).$$
The solutions of the quadratic equation $G(\gamma)=0$ are given by
$$\beta_0\geq0, \qquad \beta_{1}=\frac{|s|+|t|}{2|s|}>0.$$
Since
$$\gamma_0\leq\beta_{0}<\frac{|s|-|t|}{2|s|}\leq\beta_{1},$$
we see that $G(\gamma)\geq 0$ on $\left(\gamma_0,(|s|-|t|)/(2|s|)\right)$
if and only if $\gamma_0<\gamma\leq \beta_0$.
\end{proof}

The following two results are critical to our proof of the main theorem in
this section. They are also of some independent interest.

\begin{proposition}\label{propc2}
Let $(s, t)\in\C^2$ with $|s|^2\geq |t|^2+1$ and let $\varphi\in
L^\infty(\C, d\lambda_{\beta_0})$ with
$$\beta_0=\frac{|s|^2-|t|^2-1}{2|s|(|s|+|t|)}\geq 0.$$
Then
$$\sup_{z\in\C}\inc\left|K^{(s,t)}(z,w)\right|\,|\varphi(w)|
\exp\left(-\frac{|z|^2+|w|^2}{2}\right)\,dA(w)<\infty.$$
In particular, if $\varphi\in L^\infty(\C, d\lambda_\gamma)$ with
$\gamma<\beta_0$, then
$$\inc\left|K^{(s,t)}(z,w)\right|\,|\varphi(w)|\exp\left(-\frac{|z|^2
+|w|^2}{2}\right)\,dA(w)\leq C\, \exp\left(-\frac{\delta}{2}|z|^2\right),$$
where $C$ and $\delta$ are positive constants independent of $z$.
\end{proposition}

\begin{proof}
Note that $\beta_0<(|s|-|t|)/(2|s|)$. So for any
$\varphi\in L^\infty(\C, d\lambda_{\gamma})$ with $\gamma\leq\beta_0$,
it follows from  Lemma~\ref{lemexpc2} that
\begin{multline}\label{ineqc2}
\inc\left|K^{(s,t)}(z, w)\right||\varphi(w)|\exp\left(-\frac{|z|^2+|w|^2}{2}
\right)\,dA(w)\\
\leq  C_{\gamma}\, \|\varphi\|_{\infty,\gamma}
\exp\left(-\frac{G(\gamma)}{2|s|\left[(1-2\gamma)^2|s|^2
-|t|^2\right]}|{z}|^2\right),
\end{multline}
where $G(\gamma)$ is the continuous piecewise quadratic function
defined in Lemma~\ref{lemexpc2}. Furthermore, Lemma~\ref{lemquadrG}
implies that $G(\gamma)\geq0$ for $\gamma\leq\beta_0$. So,
for $\varphi\in L^\infty(\C, d\lambda_{\beta_0})$, we deduce
from \eqref{ineqc2} that
$$\sup_{z\in\C}\inc\left|K^{(s,t)}(z, w)\right||\varphi(w)|
\exp\left(-\frac{|z|^2+|w|^2}{2}\right)\,dA(w)\leq
C_{\beta_0}\, \|\varphi\|_{\infty,\beta_0}<\infty.$$

In the case when $\gamma<\beta_0$, let
$\gamma_0, \gamma_{1}, \gamma_{2}$ be the numbers from the
proof of Lemma~\ref{lemquadrG}. Obviously, if $|s|^2>|t|^2+1$,
then $0<\gamma_0<\min\left\{\gamma_{1}, \gamma_{2}\right\}$
and $\gamma_0<\beta_0$. If $|s|^2=|t|^2+1$, then
$\gamma_0=\beta_0=\min\left\{\gamma_{1}, \gamma_{2}\right\}=0$.
It is clear from the proof of Lemma~\ref{lemquadrG} that
$G(\gamma)>0$ for $\gamma<\beta_0$. So for
$\varphi\in L^\infty(\C, d\lambda_\gamma)$ with $\gamma<\beta_0$,
we deduce from \eqref{ineqc2} again that
$$\inc\left|K^{(s,t)}(z, w)\right||\varphi(w)|
\exp\left(-\frac{|z|^2+|w|^2}{2}\right)\,dA(w)\leq
C_{\gamma}\, \|\varphi\|_{\infty,\gamma}
\exp\left(-\frac{\delta}{2}|z|^2\right),$$
where
$$\delta=\frac{G(\gamma)}{|s|\left[(1-2\gamma)^2|s|^2
-|t|^2\right]}>0.$$
This completes the proof of the proposition.
\end{proof}

\begin{proposition}\label{propc1}
Let $(s, t)\in\C^2$ with $|s|^2\geq |t|^2+1$ and let
$\varphi\in L^\infty(\C, d\lambda_{\beta_0})$ with
$$\beta_0=\frac{|s|^2-|t|^2-1}{2|s|(|s|+|t|)}\geq 0.$$
Then
$$\sup_{w\in\C}\inc\left|K^{(s,t)}(z,w)\right|\,|\varphi(w)|
\exp\left(-\frac{|z|^2+|w|^2}{2}\right)\,dA(z)<\infty.$$
In particular, if $\varphi\in L^\infty(\C, d\lambda_\beta)$
with $\beta<\beta_0$, then
$$\inc\left|K^{(s,t)}(z,w)\right|\,|\varphi(w)|\exp
\left(-\frac{|z|^2+|w|^2}{2}\right)\,dA(z)\leq
C\, \exp\left(-\frac{\delta}{2}|z|^2\right),$$
where $C$ and $\delta$ are positive constants independent of $z$.
\end{proposition}

\begin{proof}
Let $\alpha=1/2$ in \eqref{eqc1}. We obtain
$$\inc\left| K^{(s,t)}(z,w)\right|\exp\left(-\frac{|z|^2+|w|^2}2
\right)\,dA(z)\leq\frac{2\pi\sqrt{|s|}}{\sqrt{|s|^2-|t|^2}}
\exp\left[-\frac{|s|^2-|t|^2-1}{2|s|(|s|+|t|)}|w|^2\right].$$
Consequently,
\begin{multline*}
\sup_{w\in\C}\inc\left|K^{(s,t)}(z, w)\right| |\varphi(w)|\exp\left(
-\frac{|z|^2+|w|^2}{2}\right)\,dA(z)\\
\leq\frac{2\pi\sqrt{|s|}}{\sqrt{|s|^2-|t|^2}}\sup_{w\in\C}
\left[|\varphi(w)|\exp\left(-{\beta_0}|w|^2\right)\right].
\end{multline*}
An elementary argument then yields the desired result.
\end{proof}

We now consider the boundedness of our Toeplitz-type
operators. The following point-wise estimate is well known,
\begin{equation}\label{eq:fp-point-evaluation}
e^{-|z|^2/2}|f(z)|\leq\|f\|_{p},\qquad f\in F^p,\quad z\in\C.
\end{equation}
See \cite[Corollary 2.8]{Zhu1}. For $0<p\leq\infty$ we will write
$L^p_g$ for the space of all Lebesgue measurable $f$ on $\C$
such that the function $f(z)e^{-|z|^2/2}$ belongs to $L^p(\C, dA)$.
Clearly, when $T_\varphi^{(s,t)}$ is a bounded operator
from $L^p_g$ to $L^q_g$ for any $0<p,q\leq\infty$, it will also be a
bounded operator from $F^p$ to $F^q$.

\begin{theorem}\label{thmBToep}
Let $(s,t)\in\C^2$ with $|s|^2\geq|t|^2+1$ and let
$$\beta_0=\frac{|s|^2-|t|^2-1}{2|s|(|s|+|t|)}.$$
The operator $T_\varphi^{(s,t)}$ has the following properties.
\begin{enumerate}
\item[(a)] If $\varphi\in L^\infty(\C,d\lambda_{\beta_0})$, then
$T_\varphi^{(s,t)}$ is bounded from $F^p$ to $F^q$ for all
$0< p\leq q\leq\infty$.
\item[(b)] If $\varphi\in L^\infty(\C, d\lambda_\beta)$
with $\beta<\beta_0$, then $T_\varphi^{(s,t)}$ is bounded from
$F^p$ to $F^q$ for all $0< p, q\leq\infty$.
\end{enumerate}
\end{theorem}

\begin{proof}
To proof part (a), we assume $\varphi\in L^\infty(\C,
d\lambda_{\beta_0})$. Since
\begin{equation}\label{eqsubset}
    F^p\subset F^q,\qquad 0< p\leq q\leq\infty,
\end{equation}
it suffices to prove the boundedness of $T_\varphi^{(s,t)}:
F^p\to F^p$ for any $0<p\leq \infty$. It follows from Theorem A that
$$\left\|T_\varphi^{(s,t)}k_w\right\|_p\leq C,\qquad  w\in\C,$$
for some constant $C$. By a similar argument as Theorem 23 of \cite{DZ4}
(or Theorem 2.4 of \cite{LZZ}), $T_\varphi^{(s,t)}$ is bounded on $F^p$
when $0<p<1$. When $1\leq p\leq\infty$. We will actually prove the
boundedness of $T_\varphi^{(s, t)}$ on $L^p_g$.

By Proposition \ref{propc1}, it is clear that
$$c_1:=\sup_{w\in\C}\inc\left|K^{(s,t)}(z,w)\right| |\varphi(w)|
\exp\left(-\frac{|z|^2+|w|^2}{2}\right)\,dA(z)<\infty.$$
Combining this with Fubini's theorem, we have
\begin{align*}
\|T_\varphi^{(s, t)}f\|_1&=\frac{1}{2\pi}
\inc\left|T^{(s,t)}f(z)\right|e^{-\frac{|z|^2}2}\,dA(z)\\
&\leq \frac{1}{2(\pi)^2}\inc|f(w)|e^{-\frac{|w|^2}{2}}
\,dA(w)\inc\left|K^{(s,t)}(z,w)\right| |\varphi(w)|
e^{-\frac{|z|^2+|w|^2}{2}}dA(z)
\end{align*}
for any $f\in L^1_g$. Consequently,
\begin{align*}
\|T_\varphi^{(s, t)}f\|_1\leq \frac{1}{\pi}\|f\|_{1}\sup_{w\in\C}
\inc\left|K^{(s,t)}(z,w)\right| |\varphi(w)| e^{-\frac{|z|^2+|w|^2}2}
\,dA(z)=\frac{c_1}{\pi}\|f\|_1,
\end{align*}
which shows that $T_\varphi^{(s,t)}$ is bounded on $L^1_g$. On the
other hand, it is clear from Proposition \ref{propc2} that
$$c_2:=\sup_{w\in\C}\inc\left|K^{(s,t)}(z,w)\right| |\varphi(w)|
e^{-\frac{|z|^2+|w|^2}{2}}dA(w)<\infty.$$
So, for any $f\in L^\infty_g$, we have
\begin{align*}
\|T_\varphi^{(s, t)}f\|_{\infty}\leq\frac{1}{\pi}\|f\|_{\infty}
\sup_{z\in\C}\inc\left|K^{({s},t)}(z,w)\right| |\varphi(w)|
e^{-\frac{|z|^2+|w|^2}{2}}\,dA(w)= \frac{c_2}{\pi}\|f\|_{\infty},
\end{align*}
which shows that $T_\varphi^{(s,t)}$ is bounded on $L^\infty_g$.
From complex interpolation we conclude that $T_\varphi^{(s, t)}$ is
bounded on $L^p_g$ for any $1\leq p\leq\infty$.
This completes the proof of part (a).

To prove part (b) we assume that $\varphi\in L^\infty(\C, d\lambda_\beta)$
with $\beta<\beta_0$. For any $z\in\C$ and $f\in F^p$ with $0<p\leq\infty$,
it follows from \eqref{eq:fp-point-evaluation} that
\begin{align*}
e^{-|z|^2/2} \left|T_\varphi^{(s,t)}f(z)\right|
&\leq\frac{1}{\pi}\inc\bigl|K^{(s,t)}(z,w)\bigr|
|\varphi(w)||f(w)|e^{-|w|^2}e^{-|z|^2/2}\,dA(w)\\
&\leq\frac{\|f\|_{p}}{\pi}\inc|K^{(s,t)}(z,w)|\,|\varphi(w)|
e^{-\frac{|z|^2+|w|^2}{2}}\,dA(w).
\end{align*}
By Proposition~\ref{propc2}, there exist constants $C>0$ and
$\delta>0$ such that
\begin{equation}\label{eqpointest}
e^{-|z|^2/2}\left|T_\varphi^{(s,t)}f(z)\right|
\leq C\,\|f\|_{p}e^{-\delta|z|^2/2},\qquad z\in\C.
\end{equation}
So, for $0 < q < \infty$, we have
\begin{align*}
\|T_{\varphi}^{(s,t)}f\|_q^q&=\frac{q}{2\pi}\inc\left(
|T_{\varphi}^{(s,t)}f(z)| e^{-|z|^2/2} \right)^q\,dA(z)\\
&\leq\frac{q\, C^q\,\|f\|_p^q }{2\pi}\inc e^{-q\delta|z|^2/2}\,dA(z)\\
&=C^q \delta^{-1}\|f\|_p^q,
\end{align*}
from which the desired result for any $0 < q < \infty$ follows.
The case $q=\infty$ is clear from \eqref{eqpointest}.
\end{proof}

In the following, we will show all the bounded Toeplitz-type
operators in the part (b) of the above theorem are actually compact.

\begin{theorem}\label{thmCToep}
Let $(s,t)\in\C^2$ with $|s|^2\geq|t|^2+1$ and let
$$\beta_0=\frac{|s|^2-|t|^2-1}{2|s|(|s|+|t|)}.$$
If $\varphi\in L^\infty(\C, d\lambda_\beta)$ with $\beta<\beta_0$,
then $T_\varphi^{(s,t)}$ is compact from $F^p$ to $F^q$ all
$0< p,q\leq\infty$.
\end{theorem}

\begin{proof}
By \eqref{eqsubset}, it suffices for us to prove the compactness
of $T_\varphi^{(s,t)}:F^\infty\to F^q$ for any $0<q<1$.

Let $\{f_n\}$ be a bounded sequence in $F^\infty$ such that
$f_n\rightarrow 0$ uniformly on the compact sets in $\C$.
We use Theorem 2.34 of \cite{Zhu1} to represent
$f_n\in F^\infty$ as follows:
$$f_n(z)=\sum_{a\in \rho\mathbb{Z}^2} c_a(n) k_a(z),$$
where $\rho\mathbb{Z}^2$ is a square lattice in the complex
plane with sufficiently small $\rho>0$,
$\{c_a(n):a\in\rho\mathbb{Z}^2\}\in l^\infty$, the series
converges in the ``norm topology'' of $F^\infty$, and
\begin{equation}\label{eqatomic}
\sup_{a\in \rho\mathbb{Z}^2} |c_a(n)|\leq M \|f_n\|_\infty
\end{equation}
for some positive constant $M$ that is independent of $f_n$. Note
that $c_a(n)\to 0$ as $n\to \infty$ for any $a\in \rho\mathbb{Z}^2$.

Let $\varepsilon_0=\beta_0-\beta>0$. Then for any
$a\in \rho\mathbb{Z}^2$ and $w\in\C$, we have
$$|K_a(w)|\leq \exp\left(\frac{|a|^2}{2(1+\varepsilon_0)}
+\frac{(1+\varepsilon_0)|w|^2}{2}\right).$$
Consequently,
\begin{multline}\label{eqcestcom}
 \inc \left|K^{(s,t)}(z,w)\right| |\varphi(w)| |k_a(w)|
 e^{-|w|^2}\,dA(w)\\
\leq\exp\left(-\frac{\varepsilon_0|a|^2}{2(1+\varepsilon_0)}
\right)\inc\left|K^{(s,t)}(z,w)\right||\varphi(w)|
e^{\frac{(\varepsilon_0-1)|w|^2}{2}}\,dA(w).
\end{multline}
Since $\varphi\in L^\infty(\C, d\lambda_\beta)$ with
$\beta<\beta_0$, it is clear that $\varphi(w)
e^{\frac{\varepsilon_0|w|^2}{2}}\in
L^\infty(\C, d\lambda_\gamma)$ with $\gamma=
(\beta+\beta_0)/2<\beta_0$. By Proposition~\ref{propc2},
we have
$$\inc\left|K^{(s,t)}(z,w)\right| |\varphi(w)|
e^{\frac{\varepsilon_0|w|^2}{2}}e^{-\frac{|z|^2+|w|^2}{2}}
\,dA(w)\leq C\,\exp\left(-\frac{\delta}{2}|z|^2\right),$$
where $C$ and $\delta$ are positive constants independent of $z$.
Combining this with Lemma 1.12 of \cite{Zhu1}, \eqref{eqatomic},
and \eqref{eqcestcom}, we obtain
\begin{multline*}
\sum_{a\in\rho\mathbb{Z}^2}\left|c_a(n)\right|\inc\left|K^{(s,t)}(z,w)
\right||\varphi(w)| |k_a(w)|e^{-|w|^2}\,dA(w)\\
\leq C \exp\left(\frac{1-\delta}{2}|z|^2\right)
\sum_{a\in \rho\mathbb{Z}^2}\left|c_a(n)\right|\exp
\left(-\frac{\varepsilon_0}{2(1+\varepsilon_0)}|a|^2\right)<\infty.
\end{multline*}
By Fubini's theorem,
\begin{align*}
e^{-\frac{|z|^2}2}\left|T_\varphi^{(s,t)}f_n(z)\right|
&\leq\frac{1}{\pi}e^{-\frac{|z|^2}2}\inc\left|K^{(s,t)}(z,w)
\right| |\varphi(w)| |f_n(w)|e^{-|w|^2}\,dA(w)\\
&\leq\frac{1}{\pi}e^{-\frac{|z|^2}2}\sum_{a\in
\rho\mathbb{Z}^2}\left|c_a(n)\right|\inc \left|K^{(s,t)}(z,w)
\right||\varphi(w)| |k_a(w)|e^{-|w|^2}\,dA(w)\\
&\leq\frac{C}{\pi}\exp\left(\frac{-\delta}{2}|z|^2\right)
\sum_{a\in \rho\mathbb{Z}^2}\left|c_a(n)\right|
\exp\left(-\frac{\varepsilon_0}{2(1+\varepsilon_0)}|a|^2\right).
\end{align*}
So, by H\"{o}lder's inequality, we have
\begin{align*}
\|T_\varphi^{(s,t)}f_n\|_q^q &=\frac{q}{2\pi}\inc\left|
T_\varphi^{(s,t)}f_n(z)\right|^qe^{-\frac{q|z|^2}2}\,dA(z)\\
&\leq\frac{C^q}{\delta\pi^{q}}\,\sum_{a\in \rho\mathbb{Z}^2}
\left|c_a(n)\right|^q\exp\left(-\frac{q\varepsilon_0}
{2(1+\varepsilon_0)}|a|^2\right)
\end{align*}
for any $0<q<1$.

Fix any $r>0$ and split $\rho\mathbb{Z}^2$ into two parts:
$$\Gamma_r=\{a\in \rho\mathbb{Z}^2: |a|<r\},\qquad
\Gamma^\prime_r=\{a\in \rho\mathbb{Z}^2: |a|\geq r\}.$$
Since $c_a(n)\to 0$ as $n\to \infty$ for any $a\in\rho\mathbb{Z}^2$,
it follows from Lemma 1.12 of \cite{Zhu1} that
\begin{align*}
\sum_{a\in \Gamma_r} \left|c_a(n)\right|^q\exp
\left(-\frac{q\varepsilon_0}{2(1+\varepsilon_0)}|a|^2\right)
\leq\max_{a\in\Gamma_r}|c_a(n)|^q\,\sum_{a\in
\rho\mathbb{Z}^2} \exp\left(-\frac{q\varepsilon_0}
{2(1+\varepsilon_0)}|a|^2\right)\rightarrow 0
\end{align*}
as $n\rightarrow\infty$. Also, since $\{\|f_n\|_{\infty}\}$ is
bounded, it follows from \eqref{eqatomic} that
$$\sum_{a\in \Gamma^\prime_r}\left|c_a(n)\right|^q
\exp\left(-\frac{q\varepsilon_0}{2(1+\varepsilon_0)}|a|^2\right)
\leq \|f_n\|_{\infty}^q \sum_{a\in \Gamma^\prime_r}
\exp\left(-\frac{q\varepsilon_0}{2(1+\varepsilon_0)}|a|^2
\right)\rightarrow 0$$
uniformly in $n$ as $r\rightarrow\infty$. So
$\|T^{(s, t)}f_n\|_q\rightarrow 0$ as $n\rightarrow\infty$ for
any $0<q<1$. This completes the proof of the theorem.
\end{proof}

\section{Sharpness of the constant $\beta_0$}

In this section we investigate the sharpness of the constant
$\beta_0$ in Theorem B. We first need the following critical lemma.

\begin{lemma}\label{lem:sharpnesssymbol}
Let $s=|s|e^{i\theta_s}$ and $t=|t|e^{i\theta_t}$ with
$|s|^2\ge|t|^2+1$, where $\theta_s,\theta_t\in[-\pi,\pi)$. Define
$$\theta=\frac{\theta_s+\theta_t}{2}, \qquad
\eta=\frac{\theta_t-\theta_s}{2},$$
and let
\begin{equation}\label{eqxirzr}
\xi_r=e^{-i\theta}ir,\qquad z_r=e^{-i\eta}\frac{ir}{|s|+|t|},
\end{equation}
where $r>0$. For any fixed $\gamma$ satisfying
$$\beta_0=\frac{|s|^2-|t|^2-1}{2|s|(|s|+|t|)}
\leq\gamma< \frac{1}{2},$$
consider the symbol function
\begin{equation}\label{eq:counterexample-symbol-fp}
\varphi_\gamma(u)=\exp\left[\gamma \left(\im (e^{i\theta}u)
\right)^2+2i\beta_0\re (e^{i\theta}u)\im (e^{i\theta}u)
\right],\qquad u\in\C.
\end{equation}
Then we have
\begin{equation*}\label{eq:exact-growth-fp}
e^{-|z_r|^2/2}T_{\varphi_\gamma}^{(s,t)}k_{\xi_r}(z_r)
=\frac{\sqrt{2|s|}}{\sqrt{s}\sqrt{(2|s|+|t|)(h-\delta)}}
\exp\left(\frac{h\delta}{h-\delta}r^2\right),
\end{equation*}
where $\delta=\gamma-\beta_0\in[0,1/2)$ and $h$ is a positive
constant (independent of $\gamma$ and $r$) such that $h>1/2>\delta$.
\end{lemma}

\begin{proof}
With the change of variables $w=e^{i\theta}u$, we obtain
\begin{align*}
T_{\varphi_\gamma}^{(s,t)}k_{\xi_r}(z_r)&=\inc
K^{(s,t)}(z_r, u)\varphi_\gamma(u)k_{\xi_r}(u)\,d\lambda(u)\\
&=\frac{e^{-|\xi_r|^2/2}}{\sqrt{s}}\inc\exp\left(
\frac{t z_r^2}{2 {s}}-\frac{|t| \overline{w}^2}{2|s|}+
\frac{e^{i\theta}z_r\overline{w}}{{s}}
+{w\overline{e^{i\theta}\xi_r}}\right)\varphi_\gamma
(e^{-i\theta}w) \,d\lambda(w).
\end{align*}
From the definitions of
$z_r$ and $\xi_r$, we have
$$\frac{t z_r^2}{2s}=-\frac{|t|r^2}{2|s|(|s|+|t|)^2},\qquad
\frac{e^{i\theta}z_r}{s}=\frac{ir}{|s|(|s|+|t|)},\qquad
\overline{e^{i\theta}\xi_r}=-ir.$$
If we write
$$a=\frac1{|s|(|s|+|t|)},\qquad d=1+\frac{|t|}{2|s|},$$
then
\begin{multline}\label{eq:two-dimensional-integral-fp}
e^{-|z_r|^2/2} T_{\varphi_\gamma}^{(s,t)}k_{\xi_r}(z_r)\\
=\frac{e^{-(1+a)r^2/2}}{\sqrt{s}}\inc\exp
\left(-\frac{|t| \overline{w}^2}{2|s|}+
\frac{ir\overline{w}}{|s|(|s|+|t|)}-{ir w}\right)
\varphi_\gamma(e^{-i\theta}w) \,d\lambda(w).
\end{multline}
By \eqref{eq:counterexample-symbol-fp}, we have
$$\varphi_\gamma(e^{-i\theta}w)=\exp\left(\gamma y^2
+2i\beta_0xy\right),\qquad w=x+iy.$$
Since
$$2\beta_0+\frac{|t|}{|s|}=1-a,\qquad
1-\frac{|t|}{2|s|}-\gamma=\frac{1+a}{2}-\delta,$$
we obtain
\begin{multline*}
\inc\exp\left(-\frac{|t| \overline{w}^2}{2|s|}+
\frac{ir\overline{w}}{|s|(|s|+|t|)}-{ir w}\right)
\varphi_\gamma(e^{-i\theta}w) \,d\lambda(w)=\\
\frac{1}{\pi}\inr\inr\exp\left[-dx^2-\left(\frac{1+a}{2}
-\delta\right)y^2+(1+a)ry+i(1-a)x(y-r)\right]\,dx\,dy.
\end{multline*}
The integral above is absolutely convergent because
$d\geq1$ and
$$0\leq\delta<\frac12<\frac{1+a}{2}.$$

For any fixed $y\in\mathbb R$,
$$\int_{\mathbb R}\exp\left(-dx^2+i(1-a)(y-r)x
\right)\,dx=\sqrt{\frac{\pi}{d}}\,\exp
\left(-\frac{(1-a)^2(y-r)^2}{4d}\right).$$
So \eqref{eq:two-dimensional-integral-fp} becomes
\begin{multline*}
e^{-|z_r|^2/2}T_{\varphi_\gamma}^{(s,t)}
k_{\xi_r}(z_r)\\[8pt]
=\frac{e^{-(1+a)r^2/2}}{\sqrt{s}\sqrt{\pi d}}
\int_{\mathbb R}\exp\left[-\left(\frac{1+a}{2}
-\delta\right)y^2+(1+a)ry-\frac{(1-a)^2(y-r)^2}
{4d}\right]\,dy.
\end{multline*}
Write
$$h=\frac{1+a}{2}+\frac{(1-a)^2}{4d}.$$
Then $h>1/2>\delta$, $h-\delta>0$, and
\begin{align*}
-\left(\frac{1+a}{2}-\delta\right)y^2+(1+a)ry
&-\frac{(1-a)^2(y-r)^2}{4d}-\frac{1+a}{2}r^2\\
&=-(h-\delta)\left(y-\frac{h}{h-\delta}r\right)^2
+\frac{h\delta}{h-\delta}r^2.
\end{align*}
Therefore,
\begin{equation*}\label{eq:exact-growth-fp}
e^{-|z_r|^2/2}T_{\varphi_\gamma}^{(s,t)}
k_{\xi_r}(z_r)=\frac{1}{\sqrt{s}\sqrt{d(h-\delta)}}
\exp\left(\frac{h\delta}{h-\delta}r^2\right).
\end{equation*}
This completes the proof of the lemma.
\end{proof}

To prove the main result of this section, we will use
Khintchine's inequality; see \cite{Pi}. Let $\varepsilon_1,\ldots,\varepsilon_N$ be
independent Rademacher random variables on a probability
space $(\Omega, \mathcal{F}, \mathbb P)$, so
$$\mathbb P(\varepsilon_j=1)=\mathbb P(\varepsilon_j=-1)=1/2.$$
Khintchine's inequality states that, for any $0<q<\infty$, there exist
positive constants $A_q$ and $B_q$, depending only on $q$, such that
$$ A_q\left(\sum_{j=1}^{N}|a_j|^2\right)^{1/2}\leq\left( \mathbb E
\left|\sum_{j=1}^{N}\varepsilon_j a_j\right|^q \right)^{1/q}
\leq B_q \left(\sum_{j=1}^{N}|a_j|^2\right)^{1/2},$$
for every finite sequence of complex numbers $a_1,\ldots,a_N$,
where $\mathbb E$ denotes mathematical expectation. Equivalently,
the lower estimate in the above inequality can be written as
\begin{equation}\label{eq:Khintchine}
\mathbb E \left| \sum_{j=1}^{N}\varepsilon_j a_j \right|^q
\geq A_q^q \left(\sum_{j=1}^{N}|a_j|^2\right)^{q/2}.
\end{equation}
Again, the constant $A_q$ is independent of $N$
and the coefficients $a_1,\ldots,a_N$.

We are now ready to prove the main result of this section, which
shows the constant $\beta_0$ in Theorem B is sharp.

\begin{theorem}\label{thm:sharpness-beta-fp}
Let $(s, t)\in\C^2$ with $|s|^2\ge|t|^2+1$ and let
$$\beta_0=\frac{|s|^2-|t|^2-1}{2|s|(|s|+|t|)}.$$
Then
\begin{enumerate}
\item[(a)] For every $\beta>\beta_0$, there exists a symbol function
$\varphi\in L^\infty(\mathbb C,d\lambda_\beta)$
such that $T_\varphi^{(s,t)}$ is not bounded from $F^p$ to $F^q$
for any $0<p,q\leq\infty$.
\item[(b)] There exists a symbol function
$\varphi_{\beta_0}\in L^\infty(\mathbb C,d\lambda_{\beta_0})$
such that $T_{\varphi_{\beta_0}}^{(s,t)}$ is not compact from
$F^p$ to $F^q$ for any $0<p,q\leq\infty$. Furthermore, if $p>q$,
$T_{\varphi_{\beta_0}}^{(s,t)}$ is actually not bounded
from $F^p$ to $F^q$.
\end{enumerate}
\end{theorem}

\begin{proof}
To prove part (a) we fix any $\beta>\beta_0$. Since $0\leq\beta_0<1/2$,
we can choose $\gamma$ such that
$$\beta_0<\gamma<\min\left\{\beta,\frac12\right\}.$$
Let $\delta=\gamma-\beta_0>0$ and consider the function
$\varphi_\gamma$ defined in \eqref{eq:counterexample-symbol-fp}.
We have
\begin{align*}
|\varphi_\gamma(w)|e^{-\beta|w|^2}&=\exp\left[\gamma
\left(\im (e^{i\theta}w)\right)^2-\beta\left(\left(\re (e^{i\theta}w)
\right)^2+\left(\im (e^{i\theta}w)\right)^2\right)\right]\\
&=\exp\left[-\beta \left(\re (e^{i\theta}w)\right)^2-(\beta-\gamma)
\left(\im (e^{i\theta}w)\right)^2\right]\leq1,
\end{align*}
which shows $\varphi_\gamma\in L^\infty(\mathbb C,d\lambda_\beta)$.

Suppose that $T_{\varphi_\gamma}^{(s,t)}$ is bounded from $F^p$
to $F^q$ for some $0<p,q\leq\infty$. We will obtain a contradiction.
To this end, we use \eqref{eq:fp-point-evaluation} to obtain
\begin{align}\label{eq:bounded-test-fp}
e^{-|z_r|^2/2}\left|T_{\varphi_\gamma}^{(s,t)}k_{\xi_r}(z_r)
\right|\leq\left\|T_{\varphi_\gamma}^{(s,t)} k_{\xi_r}\right\|_{q}
=\|T_{\varphi_\gamma}^{(s,t)}\|_{F^p\to F^q},
\end{align}
where $\xi_r$ and $z_r$ are defined in \eqref{eqxirzr} for $r>0$.
It follows from Lemma~\ref{lem:sharpnesssymbol} that
$$e^{-|z_r|^2/2}\left|T_{\varphi_\gamma}^{(s,t)}
k_{\xi_r}(z_r)\right|=\frac{\sqrt{2}}{\sqrt{(2|s|+|t|)(h-\delta)}}
\exp\left(\frac{h\delta}{h-\delta}r^2\right)$$
for a constant $h$, independent of $\gamma$, with $h>\delta>0$.
So the right-hand side of \eqref{eq:bounded-test-fp} tends to infinity
as $r\to\infty$. This contradicts with \eqref{eq:bounded-test-fp}.
Therefore, $T_{\varphi_\gamma}^{(s,t)}$ is not bounded from
$F^p$ to $F^q$ for any $0<p,q\leq\infty$.

To prove part (b), we also use $\varphi_{\beta_0}$ from
\eqref{eq:counterexample-symbol-fp} with $\gamma=\beta_0$
and take $\xi_r$ and $z_r$ from \eqref{eqxirzr} for $r>0$.
Similarly, it follows that $\varphi_{\beta_0}\in
L^\infty(\C,d\lambda_{\beta_0})$ and
\begin{equation}\label{eq:pointwise-c}
e^{-|z_r|^2/2}\left|T_{\varphi_{\beta_0}}^{(s,t)}k_{\xi_r}(z_r)
\right|=\frac{\sqrt{2}}{\sqrt{h(2|s|+|t|)}}=:c_{s,t},\qquad r>0.
\end{equation}
In particular, the positive constant $c_{s,t}$ is independent of $r$.
This implies that
$$\left\|T_{\varphi_{\beta_0}}^{(s,t)}k_{\xi_r}\right\|_{q}\geq
c_{s,t}\nrightarrow 0,\qquad r\rightarrow\infty.$$
Note that each $k_{\xi_r}$ is a unit vector in $F^p$ and
$k_{\xi_r}(w)\to0$ (as $r\rightarrow\infty$) uniformly for $w$ in
any compact subset of $\C$ . Thus $T_{\varphi_{\beta_0}}^{(s,t)}$
is not compact from $F^p$ to $F^q$ for any $0<p,q\le\infty$.

Next we assume $0<p<q\leq\infty$ and consider the boundedness of
$T_{\varphi_{\beta_0}}^{(s,t)}$ from $F^p$ to $F^q$.
Since $\varphi_{\beta_0}\in L^\infty(\C,d\lambda_{\beta_0})$,
Theorem~\ref{thmBToep} shows that $T_{\varphi_{\beta_0}}^{(s,t)}$
is bounded on $F^\infty$. Thus the functions
$g_r=T_{\varphi_{\beta_0}}^{(s,t)}k_{\xi_r}$ belong to $F^\infty$ and
$$\left\|g_r\right\|_{\infty} \leq
 \left\|T_{\varphi_{\beta_0}}^{(s,t)}\right\|_{F^\infty\to F^\infty}
 \left\|k_{\xi_r}\right\|_{\infty}
 =\left\|T_{\varphi_{\beta_0}}^{(s,t)}\right\|_{F^\infty\to F^\infty}.$$
Therefore,
$$M:=\sup_{r>0}\left\|g_r\right\|_{\infty}<\infty.$$

\textbf{Claim.} There is a constant $\rho>0$ such that for all
$z\in\mathbb C$ and all $r>0$,
\begin{equation}\label{eq:local-lower}
 |g_r(z)|e^{-|z|^2/2}\geq \frac{c_{s,t}}2e^{-\rho^2/2},
 \qquad |z-z_r|<\rho.
\end{equation}
First of all, for any $r>0$, we define entire functions
$$G_r(v)= e^{-\overline{z_r}v-|z_r|^2/2}g_r(z_r+v), \qquad v\in\C.$$
For $|v|\leq1$ we have
$$\left| G_r(v) \right|\leq M\exp\left( -\operatorname{Re}
(\overline{z_r}v)-\frac{|z_r|^2}{2}+\frac{|z_r+v|^2}{2}\right)
= M e^{|v|^2/2}\leq M e^{1/2}.$$
Combining this with Cauchy inequalities, we obtain
\begin{align*}
 \left|G_r^\prime(v)
 \right|\leq 2\sup_{|\zeta-v|=1/2}\left| G_r(\zeta) \right| \leq2Me^{1/2}, \qquad |v|<1/2.
\end{align*}
Fix $\rho_1\in(0,1/2)$. For $|u|<\rho_1$ it follows from
$$G_r(u)- G_r(0)=\int_0^1 u\, G_r^\prime(tu)\,dt$$
that $\left|G_r(u)- G_r(0) \right|\leq 2Me^{1/2}|u|$. Let
$$\rho=\min\left\{\rho_1,\frac{c_{s,t}}{4Me^{1/2}}\right\}>0.$$
Then by the triangle inequality and \eqref{eq:pointwise-c}, we have
\begin{align*}
\left|G_r(u)\right|\geq\left|G_r(0)\right|-\left|G_r(u)- G_r(0) \right|
\geq\frac{c_{s,t}}2, \qquad |u|\leq \rho.
\end{align*}
If $|z-z_r|<\rho$, we write $z=z_r+u$ with $|u|<\rho$. Since
$$|z_r+u|^2=|z_r|^2+2\operatorname{Re}(\overline{z_r}u)+|u|^2,$$
we have
$$|g_r(z)|e^{-|z|^2/2}=\left|e^{-\overline{z_r}u-|z_r|^2/2}
g_r(z_r+u)\right|e^{-|u|^2/2}.$$
Therefore,
$$|g_r(z)|e^{-|z|^2/2}=\left|G_r(u)\right|e^{-|u|^2/2}
\geq\frac{c_{s,t}}2e^{-\rho^2/2}.$$
This proves the claim.

Now fix some $\rho$ satisfying \eqref{eq:local-lower}, pick a
sufficiently large $R>2\rho(|s|+|t|)>0$, and let $r_j=jR$ for $j\geq1$.
It follows from \eqref{eqxirzr} that
$$|\xi_{r_j}-\xi_{r_k}|=R|j-k|, \qquad |z_{r_j}-z_{r_k}|=
\frac{R}{|s|+|t|}|j-k|.$$
Thus the sequence $\{\xi_{r_j}\}_{j\geq1}$ is uniformly separated
and the disks
$$B(z_{r_j},\rho) =  \left\{z\in\mathbb C:|z-z_{r_j}|<\rho\right\},
\qquad j\geq1,$$
are pairwise disjoint. For $N\geq1$ and $\varepsilon_j\in\{-1,1\}$
consider the functions
$$f_{\varepsilon} =\sum_{j=1}^{N}\varepsilon_j k_{\xi_{r_j}}.$$
Note that
$$|k_{\xi}(z)|e^{-|z|^2/2}=e^{-|z-\xi|^2/2}.$$
Since $\{\xi_{r_j}\}_{j\geq1}$ is uniformly separated, by Theorem~3.29
of \cite{Zhu1}, there exists a constant $A_R>0$ such that
\begin{equation}\label{eq:gaussian-overlap}
\sup_{z\in\mathbb C} \sum_{j=1}^{N} e^{-|z-\xi_{r_j}|^2/2}
\leq A_R.\end{equation}
The constant $A_R$ is independent of $N$. 
If $0<p\leq1$, H\"{o}lder's inequality yields
\begin{align*}
\left\|f_{\varepsilon}\right\|_{p}^p=\frac{p}{2\pi}\inc
|f_{\varepsilon}(z)|^p e^{-p|z|^2/2}\,dA(z)\leq
\frac{p}{2\pi}\sum_{j=1}^{N}\inc e^{-p|z-\xi_{r_j}|^2/2} dA(z)=N.
\end{align*}
Therefore, $\left\|f_{\varepsilon}\right\|_{F^p} \leq N^{1/p}$
for $0<p\leq1$.

In the case when $1<p<\infty$, it follows from
\eqref{eq:gaussian-overlap} that
\begin{align*}
|f_{\varepsilon}(z)|e^{-|z|^2/2}\leq\sum_{j=1}^{N}
e^{-|z-\xi_{r_j}|^2/2}\leq A_R,
\end{align*}
from which we deduce that
\begin{align*}
|f_{\varepsilon}(z)|^p e^{-p|z|^2/2} \leq A_R^{p-1}
\sum_{j=1}^{N}e^{-|z-\xi_{r_j}|^2/2}.
\end{align*}
Therefore,
\begin{align*}
\left\|f_{\varepsilon}\right\|_{p}^p\leq\frac{pA_R^{p-1}}{2\pi}
\sum_{j=1}^{N}\inc e^{-|z-\xi_{r_j}|^2/2}\,dA(z)=pA_R^{p-1}N.
\end{align*}
Equivalently,
$$\left\|f_{\varepsilon}\right\|_{p}\leq p^{1/p}A_R^{1-1/p}N^{1/p},
\qquad 1<p<\infty.$$

Finally, for $p=\infty$, \eqref{eq:gaussian-overlap} gives
$\left\|f_{\varepsilon}\right\|_{\infty} \leq A_R$. In summary,
we obtain for all $0<p\leq +\infty$ that
\begin{equation}\label{eq:synthesis-upper}
\left\|f_{\varepsilon}\right\|_{p}\leq C_{p,R} N^{1/p},
\end{equation}
where $C_{p,R}>0$ is independent of $N$ and the signs
$\varepsilon_1,\ldots,\varepsilon_N$. Here, we use the convention
$N^{1/\infty}=1$.

On the other hand, for any choice of
$\varepsilon_1,\ldots,\varepsilon_N$, we have
$$T_{\varphi_{\beta_0}}^{(s,t)}f_{\varepsilon}=
\sum_{j=1}^{N}\varepsilon_j g_{r_j}.$$
According to the Khintchine inequality \eqref{eq:Khintchine}, there
exists a constant $\kappa_q>0$, depending only on $q$, such that
$$\mathbb E\left|\sum_{j=1}^{N}\varepsilon_j g_{r_j}(z)
\right|^q\geq\kappa_q\left(\sum_{j=1}^{N}|g_{r_j}(z)|^2
\right)^{q/2}$$
for all $z\in\C$. Therefore, by Fubini's theorem,
\begin{align*}
\mathbb E\left\|T_{\varphi_{\beta_0}}^{(s,t)}f_{\varepsilon}
\right\|_{q}^q&=\frac{q}{2\pi}\inc\mathbb E\left|
\sum_{j=1}^{N}\varepsilon_j g_{r_j}(z)\right|^q
e^{-q|z|^2/2}\,dA(z)\\
&\geq\frac{q\kappa_q}{2\pi}\inc\left(\sum_{j=1}^{N}
|g_{r_j}(z)|^2\right)^{q/2}e^{-q|z|^2/2}\,dA(z).
\end{align*}
For every $z\in B(z_{r_j},\rho)$ with $1\leq j\leq N$, we have
$$\left(\sum_{k=1}^{N}|g_{r_k}(z)|^2\right)^{q/2}
\geq |g_{r_j}(z)|^q.$$
Since the disks $B(z_{r_j},\rho)$ are pairwise disjoint,
we further obtain
\begin{align*}
\mathbb E\left\|T_{\varphi_{\beta_0}}^{(s,t)}
f_{\varepsilon}\right\|_{q}^q&\geq\frac{q\kappa_q}{2\pi}
\sum_{j=1}^{N}\int_{B(z_{r_j},\rho)}
|g_{r_j}(z)|^q e^{-q|z|^2/2}\,dA(z).
\end{align*}
Let $c=(c_{s,t}/2)e^{-\rho^2/2}$, which only dependents on
$s$ and $t$. By \eqref{eq:local-lower}, we have
$$|g_{r_j}(z)|^q e^{-q|z|^2/2}\geq c^q,\qquad
z\in B(z_{r_j},\rho).$$
Hence
\begin{align*}
\mathbb E\left\|T_{\varphi_{\beta_0}}^{(s,t)}f_{\varepsilon}
\right\|_{q}^q\geq\frac{q\kappa_q}{2\pi}\sum_{j=1}^{N}
c^q\,\pi\rho^2=\frac{q\kappa_q c^q\rho^2}{2}N.
\end{align*}
It follows that, for each $N$, there is at least one set of signs
$\varepsilon_1,\ldots,\varepsilon_N$ such that
\begin{equation}\label{eq:output-lower}
\left\|T_{\varphi_{\beta_0}}^{(s,t)}f_{\varepsilon}
\right\|_{q}\geq\left(\frac{q\kappa_q c^q\rho^2}{2}
\right)^{1/q}N^{1/q}.
\end{equation}
Suppose, contrary to the assertion, that
$T_{\varphi_{\beta_0}}^{(s,t)}:F^p\rightarrow F^q$
is bounded for some $0<q<p\leq\infty$. For the choice of signs
in \eqref{eq:output-lower}, the boundedness
of $T_{\varphi_{\beta_0}}^{(s,t)}$ and \eqref{eq:synthesis-upper}
imply that
$$\left(\frac{q\kappa_q c^q\rho^2}{2}\right)^{1/q}N^{1/q}
\leq\left\|T_{\varphi_{\beta_0}}^{(s,t)}\right\|_{F^p\to F^q}
C_{p,R}N^{1/p}.$$
This is impossible as $N\to\infty$, because $q<p$. Therefore,
$T_{\varphi_{\beta_0}}^{(s,t)}:F^p\to F^q$ is not bounded
whenever $0<q<p\leq \infty$. This proves part (b) and
completes the proof of the theorem.
\end{proof}

\begin{remark}\label{remsharp}
In fact, we have also proved the sharpness of the constant $\beta_0$
in Theorem A in the sense that, for any sufficient small $\varepsilon>0$,
there exist $\varphi\in L^\infty(\C, d\lambda_{\beta_0+\varepsilon})$
and a sequence $\{z_n\}$ in $\C$ such that
$$\lim_{|z_n|\rightarrow\infty}\ \exp\left(-\frac{1}{2}|z_n|^2\right)
\left\|T_\varphi^{(s,t)}K_{z_n}\right\|_{p}=\infty,$$
and there exist $\varphi\in L^\infty(\C, d\lambda_{\beta_0})$ and a
sequence $\{z_n\}$ in $\C$ such that
$$\exp\left(-\frac{1}{2}|z_n|^2\right)\left\|T_\varphi^{(s,t)}K_{z_n}
\right\|_{p}\nrightarrow0$$
as $n\to\infty$.
\end{remark}


\begin{proposition}\label{propbetasharp}
Let $(s, t)\in\C^2$ with $|s|^2\ge|t|^2+1$, $0<p<\infty$, and
$$\beta_0=\frac{|s|^2-|t|^2-1}{2|s|(|s|+|t|)}.$$
For $\varphi(z)=\exp({\beta}|z|^2)$ with $\beta< 1-|t|/(2|s|)$,
there exists a sequence $\{z_n\}$ in $\C$ such that
$$\lim_{|z_n|\rightarrow\infty}\ \exp\left(-\frac{1}{2}|z_n|^2\right)
\left\|T_\varphi^{(s,t)}K_{z_n}\right\|_{p}=\infty$$
if and only if
$$\beta>\frac{2\beta_0}{1+\sqrt{1-2\beta_0}}.$$
\end{proposition}

\begin{proof}
Fix $0<p<\infty$ and take $\alpha=1-\beta$. It is clear that
$\alpha>|t|/(2|s|)$. By Lemma \ref{lemiTK}, we have
\begin{align*}
\left\|T_\varphi^{(s,t)}K_{z}\right\|_{p}=\frac{1}{\alpha}
\left\|K^{(s,t)}_{z/\alpha}\right\|_{p},\qquad z\in\C.
\end{align*}
Combining this with Proposition 17 of \cite{DZ4}, we obtain
\begin{align}\label{eqTvaKz}
\left\|T_\varphi^{(s,t)}K_{z}\right\|_{p}
&=C_\alpha\,\exp\left(\frac{|z|^2}{2\alpha^2(|s|^2-|t|^2)}\right)
\left|\exp\left(-\frac{t(|s|^2-|t|^2-1)}{2\alpha^2\overline{s}
(|s|^2-|t|^2)}{z^2}\right)\right|\nonumber\\
&\leq C_\alpha\,\exp\left(\frac{1-2\beta_0}{2\alpha^2}|z|^2\right),
\end{align}
where
$$C_\alpha=\frac{|s|^{(1/p-1/2)}}{\alpha(|s|^2-|t|^2)^{1/(2p)}}.$$
Furthermore, equality is attained if and only if $stz^2=-|stz^2|$.

Consider the quadratic function
$$H(\beta)=-\beta^2+2\beta-2\beta_0.$$
Since $0\leq \beta_0<1/2$, the equation $H(\beta)=0$ has two
real solutions:
$$\beta_{1}=1-\sqrt{1-2\beta_0},\qquad
\beta_{2}=1+\sqrt{1-2\beta_0}.$$
Simple calculations show that $\beta_{1}<1-|t|/(2|s|)\leq1<\beta_{2}$.
It follows that $H(\beta)\leq 0$ for $\beta\leq\beta_{1}$ and
$H(\beta)> 0$ for $\beta_{1}<\beta< 1-|t|/(2|s|)$.
So, for any $\beta\leq  \beta_{1}$, \eqref{eqTvaKz} gives
\begin{align*}
\exp\left(-\frac{1}{2}|z|^2\right)\left\|T_\varphi^{(s,t)}K_{z}
\right\|_{p}&\leq C_\alpha\,\exp\left(\frac{(1-2\beta_0-\alpha^2)|z|^2}
{2\alpha^2}\right)\\
&=C_\alpha\,\exp\left(\frac{H(\beta)|z|^2}{2\alpha^2}\right)
\leq C_\alpha,\qquad z\in\C.
\end{align*}
If $\beta_{1}<\beta< 1-|t|/(2|s|)$, we let $\{z_n\}$ be a sequence of
nonzero complex numbers with $stz_n^2=-|stz_n^2|$ such that the
sequence $\{|z_n|\}$ is increasing and tends to $\infty$.
Then equality occurs in \eqref{eqTvaKz} for $\{z_n\}$ and we have
$$\exp\left(-\frac{1}{2}|z_n|^2\right)\left\|T_\varphi^{(s,t)}K_{z_n}
\right\|_{p}=C_\alpha\,\exp\left(\frac{H(\beta)|z_n|^2}{2\alpha^2}
\right)\rightarrow \infty$$
as $n\rightarrow\infty$.
\end{proof}

\begin{proposition}\label{propbetasharpinfty}
Let $(s, t)\in\C^2$ with $0\leq|s|^2-|t|^2-1\leq(|s|-|t|)^2$
and let $$\gamma_1=\frac{|s|^2-|t|^2-1}{2|s|(|s|-|t|)}.$$
For $\varphi(z)=\exp({\gamma}|z|^2)$ with
$$\frac{2\gamma_1}{1+\sqrt{1-2\gamma_1}}<\gamma< 1-|t|/(2|s|),$$
there exists a sequence $\{z_n\}$ in $\C$ such that
$$\lim_{|z_n|\rightarrow\infty}\ \exp\left(-\frac{1}{2}|z_n|^2\right)
\left\|T_\varphi^{(s,t)}K_{z_n}\right\|_{\infty}=\infty.$$
\end{proposition}

\begin{proof}
Write $\alpha=1-\gamma$. Then $\alpha>|t|/(2|s|)$.
By Lemma \ref{lemiTK}, we have
\begin{align*}
\left\|T_\varphi^{(s,t)}K_{z}\right\|_{\infty}=\frac{1}{\alpha}
\left\|K^{(s,t)}_{z/\alpha}\right\|_{\infty},\qquad z\in\C.
\end{align*}
It follows from Corollary 19 of \cite{DZ4} that there exists a
sequence $\{z_n\}$ in $\C$ such that $stz_n^2=|s||t||z_n|^2$ and
\begin{align*}
\exp\left(-\frac{1}{2{\alpha}^2}|{z_n}|^2\right)
\left\|K^{(s,t)}_{z_n/\alpha}\right\|_{\infty}\geq
\frac{1}{\sqrt{|s|}}\exp\left[-\frac{|s|^2-|t|^2-1}
{2|s|(|s|-|t|)\alpha^2}|z_n|^2\right]
\end{align*}
for all $n\ge1$. Consequently,
$$\exp\left(-\frac{1}{2}|z_n|^2\right)\left\|T_\varphi^{(s,t)}
K_{z_n}\right\|_{\infty}\geq\frac{1}{\alpha\sqrt{|s|}}\exp\left[
\frac{1-2\gamma_1-\alpha^2}{2\alpha^2}|z_n|^2\right].$$

Just as in the proof of Proposition~\ref{propbetasharp},
we consider the function
$$G(\gamma)=-\gamma^2+2\gamma-2\gamma_1.$$
Since $0\leq|s|^2-|t|^2-1\leq(|s|-|t|)^2$, it is clear that
\begin{equation}\label{eqgamma1}
    \gamma_1=\frac{|s|^2-|t|^2-1}{2|s|(|s|-|t|)}\leq
    \frac{|s|-|t|}{2|s|}\leq\frac{1}{2}.
\end{equation}
The solutions of $G(\gamma)=0$ are again real and given by
$$\gamma_2=1-\sqrt{1-2\gamma_1},\qquad
\gamma_3=1+\sqrt{1-2\gamma_1}.$$
By \eqref{eqgamma1}, we have $$\gamma_2-\left(1-\frac{|t|}{2|s|}
\right)=-\sqrt{1-2\gamma_1}+\frac{|t|}{2|s|}\leq
\frac{\sqrt{|t|}(\sqrt{|t|}-2\sqrt{|s|})}{2|s|}\leq0.$$
It is clear that $\gamma_2=1-|t|/(2|s|)$ if and only if
$$\gamma_{1}=\frac{|s|^2-|t|^2-1}{2|s|(|s|-|t|)}
=\frac{|s|-|t|}{2|s|}\qquad {\rm and} \qquad t=0,$$
But this is impossible. So we have
$$\gamma_2<1-|t|/(2|s|)\leq1<\gamma_3.$$
Since $H(0)=-\gamma_1\leq0$, we have
$G(\gamma)\leq 0$ for $0\leq\gamma\leq  \gamma_2$ and
$G(\gamma)>0$ for $\gamma_2<\gamma< 1-|t|/(2|s|)$.
Thus for any $\gamma_2<\gamma< 1-|t|/(2|s|)$ we have
$$\exp\left(-\frac{1}{2}|z_n|^2\right)\left\|T_\varphi^{(s,t)}
K_{z_n}\right\|_{\infty}\geq\frac{1}{\alpha\sqrt{|s|}}\,\exp
\left[\frac{G(\gamma)|z_n|^2}{2\alpha^2}\right]
\rightarrow\infty$$
as $n\rightarrow\infty$.
\end{proof}

%

As an application, we completely characterize the boundedness and
compactness of a class of Toeplitz-type operators whose symbols
are radial exponential functions.

\begin{corollary}\label{corTvar0}
Let $(s, t)\in\C^2$ with $|s|^2=|t|^2+1$ and
$$\varphi(z)=\exp({\beta}|z|^2),\qquad \beta< 1-|t|/(2|s|).$$
For $0<p,q\leq\infty$ we have the following.
\begin{enumerate}
\item[(a)] If $\beta>0$, the operator $T_\varphi^{(s,t)}$ is not
bounded from $F^p$ to $F^q$;
\item[(b)] If $\beta=0$, the operator $T_\varphi^{(s,t)}=T^{(s,t)}$
is not compact from $F^p$ to $F^q$ and it is bounded from $F^p$
to $F^q$ if and only if $p\leq q$;
\item[(c)] If $\beta<0$, the operator $T_\varphi^{(s,t)}$ is
compact from $F^p$ to $F^q$.
\end{enumerate}
\end{corollary}

\begin{proof}
Let $\beta_0, \gamma_1$ be the constants defined in
Propositions  \ref{propbetasharp} and \ref{propbetasharpinfty}, respectively.
Since $|s|^2=|t|^2+1$, it is clear that $\beta_0=\gamma_1=0$.
If $\beta>0$, then by Propositions \ref{propbetasharp} and \ref{propbetasharpinfty},
there exists a sequence $\{z_n\}$ such that
$$\left\|T_\varphi^{(s,t)}k_{z_n}\right\|_{q}=
\exp\left(-\frac{1}{2}|z_n|^2\right)\left\|T_\varphi^{(s,t)}
K_{z_n}\right\|_{p}\rightarrow\infty,\qquad 0<p\le\infty,$$
as $n\rightarrow\infty$. This proves part (a).

Let $\varphi_{\beta_0}$ be defined as in
Theorem~\ref{thm:sharpness-beta-fp}. Since $\beta_0=0$, it is
clear that $\varphi_{\beta_0}\equiv 1$. So parts (b) and (c)
follow from Theorems~\ref{thmBToep}, \ref{thmCToep},
and \ref{thm:sharpness-beta-fp}.
\end{proof}

\section{Toeplitz-type operators in Schatten classes}

Let $\mathcal{H}$ be a separable Hilbert space. For $0 < p < \infty$,
the Schatten $p$-class, denoted by $S_p$, consists of all compact
operators $T$ on $\mathcal{H}$ such that
$$|T|x :=(T^*T)^{1/2}x= \sum_{j=1}^\infty \lambda_j \langle x, e_j
\rangle_\mathcal{H}e_j$$
for all $x \in \mathcal{H}$, where $\{e_j\}$ is an orthonormal set
in $\mathcal{H}$ and $\{\lambda_j\}$ is a non-increasing sequence
in $l^p$ of positive numbers. The numbers $\lambda_j$ 
are usually called
the singular values of $T$.
If $T\in S_p$ with singular values $\{\lambda_j\}$, we will write
$$\|T\|_{S_p} = \left( \sum_{j=1}^\infty \lambda_j ^p
\right)^{\frac{1}{p}}.$$
In particular, $S_1$ is the trace class, and if $T \in S_1$ with
$T \ge 0$, then
$$\|T\|_{S_1} = \tr (T) = \sum_{j=1}^\infty \left\langle T e_j,
e_j \right\rangle_\mathcal{H}$$
for any orthonormal basis $\{e_j\}$ of $\mathcal{H}$.
When $p=2$, the space $S_2$ is usually called the Hilbert-Schmidt
class, and
$$\|T\|^2_{S_2}=\sum_{j=1}^\infty\|Te_j\|^2$$
for any orthonormal basis $\{e_j\}$ of $\mathcal H$.

It is known that $T \in S_p$ if and only if $T^*T \in S_{\frac{p}{2}}$,
and the norm $\|T\|_{S_p}^p$ equals the trace of
$(T^*T)^{\frac{p}{2}}$, i.e.,
\begin{equation}\label{eqSp}
\|T\|_{S_p}^p = \tr\left((T^*T)^{\frac{p}{2}}\right)
=\sum_{j=1}^\infty \left\langle (T^*T)^{\frac{p}{2}} e_j,
e_j \right\rangle_\mathcal{H}.
\end{equation}
We refer the reader to \cite{Zhu2} for more information about
Schatten classes.

Recall that if $T: F^2\to F^2$ is a (possibly unbounded) linear operator
whose domain contains all kernel functions $K_z$, then the function
$$(z,w)\in\C^2\mapsto\langle Tk_z, k_w\rangle$$
is called the bivariate Berezin transform of $T$. The restriction of the
bivariate Berezin transform to the diagonal, namely, the function
$$z\in\C\mapsto\langle Tk_z, k_z\rangle$$
is called the Berezin transform of $T$. The following result gives an
integral expression for the bivariate Berezin transform of
$T_\varphi^{(s,t)}$.

\begin{proposition}\label{propBT}
Let $(s, t)\in\C^2$ with $|s|^2\geq|t|^2+1$ and
$$\beta_0=\frac{|s|^2-|t|^2-1}{2|s|(|s|+|t|)}.$$
Then for any $\varphi\in L^\infty(\C,d\lambda_{\beta_0})$,
$$\langle T_\varphi^{(s,t)}k_w, k_z\rangle=\inc \varphi(u)
K^{(s,t)}(z,u)K(u,w)e^{-\frac{|z|^2+|w|^2}{2}}d\lambda(u).$$
In particular, the Berezin transform of the bounded linear operator
$T_\varphi^{(s,t)}$ is given by
$$\widetilde{T_\varphi^{(s,t)}}(z)=e^{-|z|^2}\inc \varphi(u)
K^{(s,t)}(z,u)K(u,z)d\lambda(u).$$
\end{proposition}

\begin{proof}
By Theorem B, we have $T_\varphi^{(s,t)}k_w\in F^2$. Consequently,
$$\langle T_\varphi^{(s,t)}k_w, k_z\rangle=e^{-\frac{|z|^2
+|w|^2}{2}} \langle T_\varphi^{(s,t)} K_w, K_z\rangle
=e^{-\frac{|z|^2+|w|^2}{2}}\left(T_\varphi^{(s,t)}K_w\right)(z).$$
The desired result is then obvious.
\end{proof}

We are now able to prove the following theorem, which shows that
each compact Toeplitz-type operator that appeared in
Theorem~\ref{thmCToep} actually belongs to the Schatten class $S_p$
for all $p>0$. Our proof here employs techniques similar to those used
in the proofs of \cite[Theorems 10]{DZ4}.

\begin{theorem}\label{thmToeptr}
Let $(s, t)\in\C^2$ with $|s|^2\geq|t|^2+1$ and let
$\varphi\in L^\infty(\C, d\lambda_\beta)$ with
$$\beta<\beta_0=\frac{|s|^2-|t|^2-1}{2|s|(|s|+|t|)}.$$
Then $T_\varphi^{(s,t)}$ on $F^{2}$ belongs to the Schatten
$p$-class for all $p>0$. Moreover,
$$\tr\left(T_\varphi^{(s,t)}\right)=\inc \varphi(w) K^{(s,t)}(w,w)
\,d\lambda(w).$$
\end{theorem}

\begin{proof}
Fix any $p>0$ and write
$$\left|T_\varphi^{(s,t)}\right|=\sqrt{\left(T_\varphi^{(s,t)}
\right)^*T_\varphi^{(s,t)}}\qquad {\rm and}\qquad S
=\sqrt{\left|T_\varphi^{(s,t)}\right|^{{p}}}.$$
For any orthonormal basis $\{e_n\}$ of $F^2$, it follows from
\eqref{eqSp} that
\begin{align*}
\|T_\varphi^{(s,t)}\|_{S_p}^p=\sum_{n=0}^{\infty}\left\langle
\left|T_\varphi^{(s,t)}\right|^{p} e_n,e_n\right\rangle=
\sum_{n=0}^{\infty}\inc\left| \left \langle S e_n,K_z\right
\rangle\right|^2\,d\lambda(z).
\end{align*}
Then by Fubini's theorem and Parseval's identity, we obtain
\begin{align*}
\|T_\varphi^{(s,t)}\|_{S_p}^p&=\inc \sum_{n=0}^{\infty}
\left| \left \langle S K_z, e_n\right\rangle\right|^2d\lambda(z)
=\inc \|SK_z\|^2d\lambda(z)\\
&= \inc  \left\langle \left|T_\varphi^{(s,t)}\right|^{p} k_z,
k_z\right\rangle e^{|z|^2}d\lambda(z).
\end{align*}
If $0< p\leq2$, it follows from Lemma 3.4 of \cite{Zhu1} that
\begin{align*}
\|T_\varphi^{(s,t)}\|_{S_p}^p\leq \inc  \left\langle
\left|T_\varphi^{(s,t)}\right|^2 k_z,k_z\right
\rangle^{\frac{p}{2}} e^{|z|^2}\,d\lambda(z)
=\frac{1}{\pi}\inc \|T_\varphi^{(s,t)}K_z\|_{2}^p
e^{\frac{-p|z|^2}{2}}\,dA(z).
\end{align*}
Combining this with Theorem A, we see that $T_\varphi^{(s,t)}
\in S_p$ for any $0< p\leq2$. Since $S_{p_1}\subset S_{p_1}$
for $0<p_1<p_2<\infty$, we conclude that $T_\varphi^{(s,t)}
\in S_p$ for any $p>2$ as well.

For the trace formula, by Proposition 3.3 of \cite{Zhu1} and
Proposition~\ref{propBT}, we have
\begin{align*}
\tr\left(T_\varphi^{(s,t)}\right)=\frac{1}{\pi}\inc
\widetilde{T_\varphi^{(s,t)}}(z)\,dA(z)=\inc\left[\inc \varphi(w)
K^{(s,t)}(z,w)K_z(w)d\lambda(w)\right]\,d\lambda(z).
\end{align*}
By Proposition \ref{propc2} and the fact that $|K_z(w)|\le
\exp\left({|z|^2}/{2}+{|w|^2}/{2}\right)$, there exist positive constants
$\delta$ and $C$ such that
\begin{align*}
\inc&\left[\inc |\varphi(w)| \left|K^{(s,t)}(z,w)\right||K_z(w)|
\,d\lambda(w)\right]\,d\lambda(z)\\
&\leq\frac{1}{\pi^2} \inc\left[\inc|\varphi(w)| \left|K^{(s,t)}(z,w)
\right|e^{-\frac{|z|^2+|w|^2}{2}}\,dA(w)\right]\,dA(z)\\
&\leq \frac{C}{\pi^2}\inc\exp\left(-\frac{\delta}{2}|z|^2
\right)\,dA(z)<\infty.
\end{align*}
Since
$$\overline{K^{(s,t)}(z,w)}=\frac{\sqrt{\overline s}}
{\overline{\sqrt s}}\,K^{(\overline s,-t)}(w,z),$$
an application of Fubini's theorem and \eqref{eqTKw} yields
\begin{align*}
\tr\left(T_\varphi^{(s,t)}\right)&=\frac{\sqrt{\overline s}}
{\overline{\sqrt s}}\inc \varphi(w)\left[\overline{\inc K^{(\overline s,-t)}(w,z)K_w(z)\,d\lambda(z)}\right] d\lambda(w)\\
&=\inc \varphi(w) K^{(s,t)}(w,w) d\lambda(w).
\end{align*}
This completes the proof of the theorem.
\end{proof}

Clearly, Theorem B follows from Theorems~\ref{thmBToep},
\ref{thmCToep}, \ref{thm:sharpness-beta-fp} and \ref{thmToeptr}.

\section{Toeplitz-type operators with bounded symbols}

Recall that, for $\varphi\in L^\infty(\mathbb C,dA)$, the
classical Toeplitz operator $T_\varphi$ on $F^{2}$ is defined by
$T_\varphi f=P(\varphi f)$. Since $P=T^{(1,0)}$, it follows from
Theorem B that $T_\varphi$ is a bounded linear operator on $F^{p}$
for any $0<p<\infty$.

\begin{proposition}\label{prop:Toeplitz-old-new}
Let $(s, t)\in\C^2$ with $|s|^2\geq|t|^2+1$, $1\leq p\leq\infty$,
and $\varphi\in L^\infty(\C, dA)$. Then on $F^p$,
$$T_\varphi^{(s,t)}=T^{(s,t)}T_\varphi.$$
\end{proposition}

\begin{proof}
We first justify the application of Fubini's theorem. Write
$\rho=|t|/|s|<1$ and choose $\varepsilon$ such that
$\frac{1}{2}<\varepsilon<1-\frac{\rho}{2}$. For each
fixed $z\in\C$, we have
\begin{multline*}
\left|K^{(s,t)}(z,u)K_w(u)\right|\exp\left({-|u|^2-|w|^2}
\right)\\\leq C_z\exp\left[-\left(1-\frac{\rho}{2}\right)|u|^2
+|u||w|+\frac{|z|}{|s|}|u|-|w|^2\right]
\end{multline*}
for some constant $C_z>0$. Since
$$|u||w| \leq \varepsilon |u|^2+\frac{1}{4\varepsilon}|w|^2,
\qquad\frac{|z||u|}{|s|} \leq \frac{2-\rho-2\varepsilon}{4}
|u|^2+\frac{|z|^2}{(2-\rho-2\varepsilon)|s|^2},$$
we obtain
$$\left|K^{(s,t)}(z,u)K_w(u)\right|\exp\left({-|u|^2-|w|^2}\right)
\leq C_{z,\varepsilon}\exp\left(-\delta|u|^2-\eta|w|^2\right),$$
where $C_{z,\varepsilon}=C_z\exp\left(|z|^2/[(2-\rho-2\varepsilon)|s|^2]\right)$,
and
$$\delta=\frac{1}{2}\left(1-\frac{\rho}{2}-\varepsilon\right)>0,
\qquad \eta=1-\frac{1}{4\varepsilon}>\frac{1}{2}.$$
Fix any $1\leq p\leq\infty$ and $f\in F^p$. We easily obtain
$$\inc|f(w)|e^{-\eta|w|^2}\,dA(w)\leq C_{p,\eta}\|f\|_{p}$$
for some constant $C_{p,\eta}>0$. Consequently,
\begin{multline*}
\int_{\mathbb C}\inc\left|K^{(s,t)}(z,u)K_w(u)\varphi(w)f(w)
\right|\,d\lambda(w)\,d\lambda(u)\\
\leq C_{z,\varepsilon}\|\varphi\|_{\infty,0}\left(\int_{\mathbb C}
e^{-\delta|u|^2}\,dA(u)\right)\left(\int_{\mathbb C}|f(w)|
e^{-\eta|w|^2}\,dA(w)\right)<\infty.
\end{multline*}
From the integral representation
$$T_\varphi f(u)=\inc K_w(u)\varphi(w)f(w)\,d\lambda(w),$$
Fubini's theorem, and \eqref{eqTKw}, we obtain
\begin{align*}
T^{(s,t)}T_\varphi f(z)&=\inc K^{(s,t)}(z,u)\left(\inc
K_w(u)\varphi(w)f(w)\,d\lambda(w)\right)\,d\lambda(u)\\
&=\inc\left(\inc K^{(s,t)}(z,u)K_w(u)\,d\lambda(u)
\right)\varphi(w)f(w)\,d\lambda(w)\\
&=\inc K^{(s,t)}(z,w)\varphi(w)f(w)\,d\lambda(w)
=T_\varphi^{(s,t)}f(z),
\end{align*}
which proves the desired result.
\end{proof}

In view of the above proposition, we will consider the compactness
of Toeplitz-type operators with bounded symbols on $F^p$ for
$1\leq p\leq\infty$. Recall from Theorem~B of \cite{DZ4} that if
$|s|^2>1+|t|^2$, then $T^{(s,t)}$ is compact on $F^p$ for all
$0<p\leq\infty$. From this and Proposition~\ref{prop:Toeplitz-old-new}
we conclude that $T_\varphi^{(s,t)}$ is compact on $F^p$ for
$1\leq p\leq\infty$ whenever $|s|^2>1+|t|^2$ and
$\varphi\in L^\infty(\C, dA)$. Note that this conclusion also follows
directly from Theorem ~\ref{thmCToep}.

When $|s|^2=1+|t|^2$, we first consider the case $p=2$.
In this case, by Theorem A of \cite{DZ3}, $T^{(s,t)}$ is unitary on
$F^2$. So $T_\varphi^{(s,t)}$ is compact on $F^2$ if and only if
$T_\varphi$ is compact on $F^2$, which is known to be equivalent
to the condition $\widetilde{\varphi}(z)\to0$ as $|z|\to\infty$, where
$$\widetilde{\varphi}(z)= \frac{1}{\pi}\int_{\mathbb C}
\varphi(w)e^{-|z-w|^2}\,dA(w).$$
For the case $p\neq2$, we need the following.

\begin{lemma}\label{lemTinver}
Let $(s, t)\in\C^2$ with $|s|^2=|t|^2+1$ and $1\leq p\leq\infty$.
Then $T^{(s,t)}$ is invertible on $F^p$ and
$\left[T^{(s, t)}\right]^{-1}=c\,T^{(\overline{s}, -t)}$,
where $c=|s|/(\sqrt{s}\sqrt{\overline{s}})$.
\end{lemma}

\begin{proof}
From the proof of Theorem~\ref{thmBToep}, we see that
the associated positive operator $Q^{(s,t)}$ of
$T_\varphi^{(s,t)}$ with $\varphi\equiv1$, which is defined by
$$Q^{(s,t)} f(z)=\inc \left|K^{(s,t)}(z,w)\right|f(w)\,d\lambda(w),$$
is bounded on $L^p_g$ for $1\leq p\leq\infty$ for any $(s, t)\in\C^2$
with $|s|^2=|t|^2+1$. Consequently,
$$ \inc\inc\left|K^{(s,t)}(z,u)K^{(\overline{s},-t)}(u,w)f(w)\right|
\,d\lambda(w)\,d\lambda(u)
=Q^{(s,t)} Q^{(\overline{s},-t)}(|f|)(z)<\infty.
$$
Recall from Lemma 4 of \cite{DZ3} that if $|s|^2=|t|^2+1$,
then
$$T^{(s,t)}K_w^{(\overline{s},-{t})}(z)=c\,K_w(z),\qquad z,w\in\C.$$
Combining this with Fubini's theorem, we obtain
\begin{align*}
T^{(s,t)} c T^{(\overline{s},-t)} f(z)&=c \inc K^{(s,t)}(z,u)
\left(\inc K^{(\overline{s},-t)}(u,w)f(w)\,d\lambda(w)
\right)\,d\lambda(u)\\
&=c\inc\left(\inc K^{(s,t)}(z,u)K^{(\overline{s},-t)}(u,w)
\,d\lambda(u)\right)f(w)\,d\lambda(w)\\
&=\inc K(z,w)f(w)\,d\lambda(w)=f(z)
\end{align*}
for any $f\in F^p$, which proves the desired result.
\end{proof}

Bauer and Isralowitz \cite{BI} showed that, for $1<p<\infty$, the
classical Toeplitz operator $T_\varphi$ with $\varphi\in L^\infty(\C, dA)$
is compact on $F^p$ if and only if $\widetilde{\varphi}(z)\to0$ as
$|z|\to\infty$. Recently, Fulsche \cite{Fu1} showed that, for
$\varphi\in L^\infty(\C, dA)$, the compactness of $T_\varphi$ on $F^p$
for any given $p\in[1,\infty]$ is equivalent to the compactness of
$T_\varphi$ on $F^2$. Thus we conclude from Lemma~\ref{lemTinver}
that, for any $(s, t)\in\C^2$ with $|s|^2=|t|^2+1$ and $1\leq p\leq\infty$,
$T_\varphi^{(s,t)}$ is compact on $F^p$ if and only if
$T_\varphi$ is compact on $F^p$. We summarize the analysis above
as follows.

\begin{theorem}\label{thmBSym}
Let $(s, t)\in\C^2$ with $|s|^2\geq|t|^2+1$, $1\leq p\leq\infty$, and
$\varphi\in L^\infty(\C, dA)$.
\begin{itemize}
\item[(a)] When $|s|^2>1+|t|^2$,  $T_\varphi^{(s,t)}$ is compact
on $F^p$.
\item[(b)] When $|s|^2=1+|t|^2$, $T_\varphi^{(s,t)}$ is compact
on $F^p$ if and only if $\widetilde{\varphi}(z)\to0$ as $z\to\infty$.
\end{itemize}
\end{theorem}

In the rest of this section we will obtain several estimates for the
Berezin transform of Toeplitz-type operators with bounded symbols.
We begin with the following upper bound from Theorem 5 of
\cite{DZ4} for the integral kernel $K^{(s,t)}(z,w)$: if $|s|>|t|$, then
\begin{equation}\label{eqUBKzw}
\left| K^{(s,t)}(z,w)\right|\leq c\, \exp\left[\frac{(|s|^2+1)(|z|^2
+|w|^2)}{3|s|^2-|t|^2+1}-\frac{(|s|^2-|t|^2-1)
\re({s\overline{z}w})}{|s|^2(3|s|^2-|t|^2+1)}\right]
\end{equation}
for any $z, w\in\C$, where
$$c=\frac{1}{\sqrt{|s|}}\left[\frac{2{|s|}}{\sqrt{4|s|^2-|t|^2} }
\right]^{\frac{2(4|s|^2-|t|^2)}{3|s|^2-|t|^2+1}}
\geq\frac{1}{\sqrt{|s|}}.$$
As an application, we obtain the following estimate for bivariate
Berezin transforms of Toeplitz-type operators on the Fock space.

\begin{lemma}\label{lembiBere}
Let $(s, t)\in\C^2$ with $|s|^2\geq|t|^2+1$ and
$\varphi\in L^\infty(\C,dA)$. Then on $F^{2}$,
\begin{align*}
&\left|\langle T_\varphi^{(s,t)}k_w, k_z\rangle\right|\\
&\leq\frac{c\|\varphi\|_{\infty,0}(3|s|^2-|t|^2+1)}{2|s|^2-|t|^2}
\exp\left[-\frac{(|s|^2-|t|^2-1)\left[|s|^2(|z|^2+|w|^2)+
2\re ({s\overline{z}w})\right]}{4|s|^2(2|s|^2-|t|^2)}\right],
\end{align*}
where $c\geq\frac{1}{\sqrt{|s|}}$ as in \eqref{eqUBKzw}.
\end{lemma}

\begin{proof}
Since $\varphi\in L^\infty(\C,dA)$, it follows from
Theorem~\ref{thmBToep} that $T_\varphi^{(s,t)}K_w\in F^2$.
Then we have
\begin{align*}
\left|\langle T_\varphi^{(s,t)}k_w, k_z\rangle\right|&
=e^{-\frac{|z|^2+|w|^2}{2}}\left| T^{(s,t)}(\varphi K_w)(z)
\right|\\&\leq{\|\varphi\|_{\infty,0}}
e^{-\frac{|z|^2+|w|^2}{2}}\inc\left|K^{(s,t)}(z,v)\right|
|K_w(v)|d\lambda(v).
\end{align*}
Combining \eqref{eqUBKzw} with Corollary 2.5 of \cite{Zhu1}, we obtain
\begin{align*}
&e^{-\frac{|z|^2}{2}}\inc \left|K^{(s,t)}(z,v)\right|
|K_w(v)|\,d\lambda(v)\\
&\leq ce^{-\frac{(|s|^2-|t|^2-1)|z|^2}{2(3|s|^2-|t|^2+1)}}\,
\inc e^{\frac{(|s|^2+1)|v|^2}{3|s|^2-|t|^2+1}}
\left|e^{-\frac{(|s|^2-|t|^2-1) {s\overline{z}v}}
{|s|^2(3|s|^2-|t|^2+1)}+\overline{w}v}\right|d\lambda(v)\\
&=\frac{c(3|s|^2-|t|^2+1)}{2|s|^2-|t|^2}e^{-\frac{(|s|^2-|t|^2-1)
|z|^2}{2(3|s|^2-|t|^2+1)}}\, e^{\frac{3|s|^2-|t|^2+1}
{4(2|s|^2-|t|^2)}\left|-\frac{(|s|^2-|t|^2-1) {\overline{s }z}}
{|s|^2(3|s|^2-|t|^2+1)}+{w}\right|^2}.
\end{align*}
Note that $|s|\geq1$. A few lines of elementary calculations show that
\begin{multline*}
-\frac{|w|^2}{2}+{\frac{3|s|^2-|t|^2+1}{4(2|s|^2-|t|^2)}\left|
-\frac{(|s|^2-|t|^2-1) {\overline{s }z}}{|s|^2(3|s|^2-|t|^2+1)}+{w}
\right|^2}\\
\leq{-\frac{(|s|^2-|t|^2-1)}{4(2|s|^2-|t|^2)}\left[
\frac{-(|s|^2-|t|^2-1)}{3|s|^2-|t|^2+1}|z|^2+|w|^2
+\frac{2\re ({s\overline{z}w})}{|s|^2}\right]}
\end{multline*}
and
$$-\frac{1}{3|s|^2-|t|^2+1}+\frac{|s|^2-|t|^2-1}{2(2|s|^2-|t|^2)
(3|s|^2-|t|^2+1)}=-\frac{1}{2(2|s|^2-|t|^2)}.$$
The desired result then follows.
\end{proof}

The following result gives some upper estimates for the Berezin
transform of Toeplitz-type operators with bounded symbols.

\begin{theorem}\label{thmBTBT}
Let $(s, t)\in\C^2$ with $|s|^2\geq|t|^2+1$ and
$\varphi\in L^\infty(\C,dA)$. Then the Berezin transforms of
Toeplitz-type operators on $F^{2}$ have the following properties.
\begin{enumerate}
\item[(a)] If $|s|^2>|t|^2+1$, then
$$\left|\widetilde{T_\varphi^{(s,t)}}(z)\right|\leq
\frac{c(3|s|^2-|t|^2+1)\|\varphi\|_{\infty,0}}{2|s|^2-|t|^2}
\exp\left[{-\frac{(|s|^2-|t|^2-1)\left(|s|^2+\re {s}\right)}
{2|s|^2(2|s|^2-|t|^2)}|z|^2}\right],$$
where $c$ is the constant from \eqref{eqUBKzw}.
\item[(b)] If $|s|^2=|t|^2+1$ with $(s,t)\neq (1,0)$, then
$$\left|\widetilde{T_\varphi^{(s,t)}}(z)\right|\leq
\frac{2\|\varphi\|_{\infty,0}}{\sqrt{4|s|^2-|t|^2}}
\exp\left[-\frac{(\re s-1)^2|z|^2}{2(|s|^2-\re s)+
|t|\sqrt{(\re s-1)^2+4(\im s)^2}}\right].$$
\item[(c)] If $|s|^2=|t|^2+1$ with $(s,t)=(1,0)$, then
$$\left|\widetilde{T_\varphi^{(s,t)}}(z)\right|
\leq \|\varphi\|_{\infty,0}.$$
\end{enumerate}
\end{theorem}

\begin{proof}
Part (a) clearly follows from Lemma \ref{lembiBere}.

We now assume that $|s|^2=|t|^2+1$. By
Proposition~\ref{propBT} and Lemma \ref{lemiTK}, we have
\begin{multline*}
\left|\widetilde{T_\varphi^{(s,t)}}(z)\right|
\leq\frac{\|\varphi\|_{\infty,0}}{\pi}e^{-{|z|^2}}\inc
\left|K^{(s,t)}(z, v)K_z(v)\right| \exp\left(-|v|^2\right)\,dA(v)\\
=\frac{2\|\varphi\|_{\infty,0}}{\sqrt{4|s|^2-|t|^2}}e^{-{|z|^2}}
\left|\exp\left(\frac{\overline{s}t}{2|s|^2}z^2-\frac{\overline{s}t
\left({z}+{s}{z}\right)^{2}}{2|s|^2(4|s|^2-|t|^2)}\right)\right|
\exp\left(\frac{\left|{z}+{s}{z}\right|^2}{4|s|^2-|t|^2}\right).
\end{multline*}
Since $|s|^2=|t|^2+1$, we have
\begin{align*}
&\left|\exp\left(\frac{\overline{s}t}{2|s|^2}z^2-\frac{\overline{s}t
\left(1+{s}\right)^2{z}^{2}}{2|s|^2(4|s|^2-|t|^2)}\right)\right|
\exp\left(\frac{\left|1+{s}\right|^2{|z|^2}}{4|s|^2-|t|^2}
-{|z|^2}\right)\\
&\quad=\left|\exp\left(\frac{3\overline{s}-s-2}{2(3|s|^2+1)}
t{z}^{2}\right)\right|\exp\left(-\frac{2(|s|^2-\re s)}
{3|s|^2+1}{|z|^2}\right)\\
&\quad=\exp\left(\frac{\re (s-1)\re (t{z}^{2})}{3|s|^2+1}
+\frac{2\im s\im (tz^2)}{3|s|^2+1} \right)
\exp\left(-\frac{2(|s|^2-\re s)}{3|s|^2+1}{|z|^2}\right).
\end{align*}
Note that
$$\re (s-1)\re (t{z}^{2})+2\im s\im (tz^2)\leq
\sqrt{(\re s-1)^2+4(\im s)^2}\,|t||z|^2.$$
So we have
\begin{multline}\label{eqBere}
\left|\widetilde{T_\varphi^{(s,t)}}(z)\right|\leq
\frac{2\|\varphi\|_{\infty,0}}{\sqrt{3|s|^2+1}}\\
\cdot\exp\left(-\frac{2(|s|^2-\re s)-|t|\sqrt{(\re s-1)^2+
4(\im s)^2}}{3|s|^2+1}{|z|^2}\right).
\end{multline}

To prove (b) we further assume $(s,t)\neq (1,0)$. It is clear that
$$4(x^2+y^2-x)^2-\left[(x-1)^2+4y^2\right](x^2+y^2-1)=(x-1)^2\left[3(x^2+y^2)+1\right]$$
for all real numbers $x$ and $y$.
Let $\re s=x$ and $\im s=y$. Since
$$2(|s|^2-\re s)+|t|\sqrt{(\re s-1)^2+4(\im s)^2}\neq0,$$
it follows that
\begin{multline*}
\frac{2(|s|^2-\re s)-|t|\sqrt{(\re s-1)^2+4(\im s)^2}}{3|s|^2+1}\\
=\frac{(\re s-1)^2}{2(|s|^2-\re s)+|t|\sqrt{(\re s-1)^2+4(\im s)^2}}.
\end{multline*}
Combining this with \eqref{eqBere}, we see that part (b) holds.

Part (c) clearly follows from \eqref{eqBere}. This
completes the proof of the theorem.
\end{proof}

We are now ready to prove Theorem C.

\begin{proof}[Proof of Theorem C]
Part (a) follows directly from Theorem~\ref{thmBTBT}. We now
assume that $s= 1\pm i|t|$. Consider the canonical integral operator
$T^{(s,t)}$ on $F^2$, which is $T_\varphi^{(s,t)}$ with
$\varphi\equiv1\in L^\infty(\C,dA)$. By Proposition~\ref{propBT}
and \eqref{eqTKw}, we have
\begin{align*}
\left|\widetilde{T^{(s,t)}}(z)\right|&=\frac1{\sqrt{|s|}}\exp
\left[\re \frac{tz^2-\overline{tz^2}+2|z|^2}{2s}-|z|^2\right]\\
&=\frac1{\sqrt{|s|}}\exp\left[\frac{\im s \im (tz^2)+(\re s-|s|^2)
|z|^2}{|s|^2}\right].
\end{align*}
It follows from $s= 1\pm i|t|$ that
$$\left|\widetilde{T^{(s,t)}}(z)\right|=\frac1{\sqrt{|s|}}\exp
\left[\frac{|t|}{|s|^2}\left(\pm \im (tz^2)-|tz^2|\right)\right].$$
In accordance with the $\pm$ in $s=1\pm i|t|$, we let $\{z_n\}$
be a sequence of nonzero complex numbers such that
$\pm\im (tz_n^2)=|tz_n^2|$ and such that
the sequence $\{|z_n|\}$ increases to $\infty$.
Then it is clear that $\left|\widetilde{T^{(s,t)}}(z_n)\right|
=1/\sqrt{|s|}$. This proves part (b) and
completes the proof of the theorem.
\end{proof}


\begin{thebibliography}{99}

\bibitem{B1} V. Bargmann, On a Hilbert space of analytic functions
and an associated integral transform I, \textit{Comm. Pure Appl. Math}.
\textbf{14} (1961), 187-214.
\bibitem{BI} W. Bauer and J. Isralowitz, Compactness characterization
of operators in the Toeplitz algebra of the Fock space $F_\alpha^p$,
\textit{J. Funct. Anal.} \textbf{263} (2012), 1323-1355.
\bibitem{DZ3} X. Dong and K. Zhu, Canonical integral operators on
the Fock space, \textit{Math. Z.} 2024, \textit{306}:64.
\bibitem{DZ4} X. Dong and K. Zhu, Canonical integral operators on
the Fock space II, arXiv:2509.16672.
\bibitem{Fu1} R. Fulsche, A Wiener algebra for Fock space operators,
Operator Theory, Related Fields, and Applications. IWOTA 2023.
Operator Theory: Advances and Applications {\bf 307},
Birkh\"{a}user, Cham, 2025.
\bibitem{Pi}
G. Pisier, \textit{Martingales in Banach spaces}, Cambridge Studies in Advanced Mathematics, {\bf 155}, Cambridge, 2016.
\bibitem{LZZ} Z. Lou, K. Zhu, and S. Zhu, Linear operators on Fock
spaces, \textit{Integral Equations Operator Theory}
\textbf{88} (2017), 287-300.
\bibitem{Zhu1} K. Zhu, \textit{Analysis on Fock Spaces}, Springer,
New York, 2012.
\bibitem{Zhu2} K. Zhu, \textit{Operator Theory in Function Spaces}
(2nd edition), Mathematical Surveys and Monographs {\bf 138},
American Mathematical Society, 2007.
\end{thebibliography}
\end{document}